\documentclass[10pt,reqno]{article}
\usepackage{amsmath,amssymb,amsthm}
\usepackage{graphicx,color}
\usepackage{epstopdf,color}
\DeclareMathAlphabet{\itbf}{OML}{cmm}{b}{it}

\def\bq{{{\itbf q}}}

\def\br{{{\itbf r}}}

\def\by{{{\itbf y}}}
\def\bx{{{\itbf x}}}

\def\bk{{{\itbf k}}}

\def\bu{{{\itbf u}}}
\def\bv{{{\itbf v}}}
\def\bw{{{\itbf w}}}

\def\bzeta{{\boldsymbol{\zeta}}}
\def\bxi{{\boldsymbol{\xi}}}
\def\balpha{{\boldsymbol{\alpha}}}
\def\bbeta{{\boldsymbol{\beta}}}

\def\eps{{\varepsilon}}

\newcommand{\NN}{\mathbb{N}}

\newcommand{\RR}{\mathbb{R}}

\newcommand{\EE}{\mathbb{E}}

\newtheorem{thm}{Theorem}[section]
\newtheorem{corollary}[thm]{Corollary}
\newtheorem{lemma}[thm]{Lemma}
\newtheorem{proposition}[thm]{Proposition}

\numberwithin{equation}{section}

\begin{document}  
 
 \title{Fourth-moment Analysis for Wave Propagation in the White-Noise Paraxial Regime} 

\author{Josselin
Garnier\thanks{\footnotesize Laboratoire de Probabilit\'es et
Mod\`eles Al\'eatoires \& Laboratoire Jacques-Louis Lions,
Universit\'e Paris Diderot, 75205 Paris Cedex 13, France
(garnier@math.univ-paris-diderot.fr)} 
\and Knut S\o lna\thanks{\footnotesize Department of Mathematics, 
University of California, Irvine CA 92697
(ksolna@math.uci.edu)}
}

\maketitle

\begin{abstract}
In this paper we consider the It\^o-Schr\"odinger model for wave propagation in 
random media in the paraxial regime.
We solve the equation for the fourth-order moment of the field in the regime where
the correlation length of the medium is smaller than the initial beam width. 
As applications we prove that the centered fourth-order moments of the field
satisfy the Gaussian summation rule, we derive the covariance function of the intensity of the transmitted beam,
and the variance of the smoothed Wigner transform of the transmitted field. The second application
is used  to explicitly quantify the scintillation of the transmitted beam and the third application
to quantify the statistical stability of the Wigner transform.
\end{abstract}



\section{Introduction}

%%%%%%%%%%%%%%%%%%%%%%%%%%%%

In many wave propagation scenarios the medium is not constant,
but varies in a complicated fashion on a scale that may be small compared
to the total propagation distance.  This is the case for wave propagation through the
turbulent atmosphere, the earth's crust, the ocean, and complex biological tissue for instance. 
If one aims  to use transmitted or reflected waves for communication or imaging
purposes it is important to characterize how such microstructure
affects and corrupts the wave. Such a characterization is particularly
important for modern imaging techniques such as seismic interferometry
or coherent interferometric imaging that correlate  wave field traces 
that have been strongly corrupted by the microstructure and use their
space-time correlation function for imaging. 
The wave field correlations are second-order moments of the wave field 
and a characterization of the signal-to-noise
ratio then involves a fourth-order moment calculation.

Motivated by the situation described above we consider  wave  propagation
through time-independent media with a complex 
spatially varying index of refraction that can be modeled as the realization of
a random process.  Typically 
we cannot expect to know the index of refraction pointwise, but we may be able to
characterize its statistics and we are interested in how the statistics of the medium
affects  the statistics of the wave  field.
In its most common form,
the analysis of wave propagation in random media consists in studying the field $v$ solution
of the scalar time-harmonic wave or Helmholtz equation
\begin{equation}
\label{eq:scalarwave}
\Delta v +k_0^2 n^2(z,\bx) v = 0, \quad \quad (z,\bx) \in \RR \times \RR^{2},
\end{equation}
where $k_0$ is the free space homogeneous wavenumber and $n$ is a randomly heterogeneous index of refraction.
Since the index of refraction $n$ is a random process, the field  $v$ is also a random process
whose statistical behavior can be characterized  by the calculations of its moments.
Even though the scalar wave equation is simple and linear, the relation between the statistics
of the index of refraction and the statistics of the field is highly nontrivial and nonlinear.
In this paper we consider a primary scaling regime corresponding 
to long-range beam propagation and small-scale medium fluctuations giving
negligible backscattering.
This is  the so-called white-noise paraxial regime, as described by the 
It\^o-Schr\"odinger model, which is presented  in Section \ref{sec:model}.
This model is  a simplification of the model (\ref{eq:scalarwave}) since it corresponds
to an evolution problem, but yet in the regime that we consider it describes the propagated field
in a weak sense in that it gives the correct statistical structure of the wave field.
The It\^o-Schr\"odinger model   can be derived  rigorously from (\ref{eq:scalarwave})
by a separation of scales technique in the high-frequency regime (see \cite{bailly96} in the case of a randomly layered
medium and \cite{garniers0,garniers1,garniers2} in the case of a three-dimensional random medium).
It  models  many situations, for instance
laser beam propagation \cite{strohbehn},
time reversal in random media  \cite{blomgren,PRS04},
underwater acoustics \cite{tappert},
or migration problems in geophysics \cite{claerbout85}.
The It\^o-Schr\"odinger model   allows for the use of  It\^o's stochastic calculus,  
which in turn enables the  closure of the hierarchy of moment equations \cite{fps,ishimaru}.
Unfortunately, even though  the equation for the second-order moments can be solved,
the equation for the fourth-order moments is very  difficult and only approximations
or numerical solutions are available (see  \cite{fante75,gozani,tatarski71,uscinski,whitman} 
and \cite[Sec.~20.18]{ishimaru}).

Here, we consider a secondary scaling regime
corresponding to the so-called scintillation regime and in this regime we derive explicit
expressions for the fourth-order moments.
The scintillation scenario  is a well-known paradigm, related to the observation that 
the irradiance of a star fluctuates due to interaction of the light with  the turbulent atmosphere.
This common observation is far from being fully understood mathematically.
However,  experimental observations indicate that the statistical distribution of the irradiance is exponential, with the irradiance being the square magnitude of the complex wave field. 
Indeed it is a well-accepted conjecture in the physical literature that the statistics of the complex
wave field becomes 
circularly  symmetric complex  
Gaussian when the wave propagates through the turbulent atmosphere \cite{valley,ya},
so that the irradiance is the sum of the squares of two independent real Gaussian random variables, which 
has chi-square distribution with two degrees of freedom, that is 
an exponential distribution.
However,  so far there is no mathematical proof of this conjecture, except in randomly layered media 
\cite[Chapter 9]{book1}.  The regime we consider here, which 
we refer to as the scintillation regime,  gives results
for the fourth-order moments that are consistent with the scintillation or Gaussian conjecture.
We prove in Section \ref{sec:gsr} that the incoherent zero-mean wave field (i.e. the fluctuations of the wave field
defined as the difference between the field and its expectation) has fourth-order moments that obey the Gaussian summation rule.
As a result we can discuss the statistical character of the irradiance 
 in detail in Section \ref{sec:intensity}.

 Certain functionals of the solution to the white-noise paraxial wave equation can be characterized  in 
some specific regimes  \cite{balMMS,bal,fann06,PRS07}.  An important aspect of such characterizations
is the so-called  statistical stability  property  which corresponds
to functionals of the wave field becoming deterministic in the considered scaling regime.  
This is in particular the case in the limit of rapid decorrelation of the medium  fluctuations
(in both longitudinal and lateral coordinates). 
As shown in \cite{balMMS} the statistical stability also depends on the initial data and can be lost for very rough initial data even with a high lateral diversity as considered there.
In \cite{ryzhikCMP,ryzhikDC} the authors also   consider  a situation with rapidly fluctuating 
random medium fluctuations  and a regime in which the so-called Wigner
 transform itself is statistically stable. 
The Wigner transform is known to be a convenient tool
to analyze problems involving the Schr\"odinger equation \cite{gerard,ryzhik}. 
In Section \ref{sec:wigner} we are able to push through
a detailed and quantitative analysis of the stability of this quantity using our results on the fourth-order moments.
An important aspect of our analysis
is that we are able to derive an explicit expression of the coefficient of variation of the smoothed Wigner transform
as a function of the smoothing parameters, in the general situation in which 
the standard deviation can be of the same order as the mean.  This is a realistic scenario,
we are not deep into a statistical stabilization situation, 
but in a situation where the parameters of the problem
give partly coherent but fluctuating  wave functionals.  Here we are for the first time
able to explicitly quantify such fluctuations and how their magnitude
can be controlled by smoothing of the Wigner transform.   We believe that these 
results are important for the many applications where the smoothed Wigner transform appears naturally.  
 
 The outline of the paper is as follows:  In Section \ref{sec:model} we introduce
the  It\^o-Schr\"odinger model.
In Section \ref{sec:main} we summarize our main results.
In Sections \ref{sec:mom1}-\ref{sec:mom} we describe the general equations for the moments of the field.
In Section \ref{sec:second} we discuss 
the second-order moments. In Section \ref{sec:fourth} we introduce and analyze the fourth-order moments
and the particular parameterization that will be useful to untangle these. 
In Section \ref{sec:regime} we introduce the so-called scintillation regime where
we can get an explicit characterization of the fourth-order moments via the main result of
the paper presented in  Proposition \ref{prop:sci1}.
Next we discuss three applications of the main result:
In Section \ref{sec:gsr} we prove that the centered fourth-order moments
satisfy the Gaussian summation rule,
in Section \ref{sec:intensity}
we compute the scintillation index, and in Section \ref{sec:wigner} we analyze the statistical stability of the 
smoothed Wigner transform. 

\section{The White-Noise Paraxial Model}
\label{sec:model}%
Let us consider the time-harmonic wave equation with homogeneous wavenumber $k_0$,
random index of refraction $n(z,\bx)$, and source in the plane $z=0$:
\begin{equation}
\label{eq:helm}
\Delta v + k_0^2 n^2 (z,\bx) v = - \delta(z) f(\bx) \, ,
\end{equation}
for $\bx \in \RR^2$ and $z\in [0,\infty)$.
Denote by $\lambda_0$ the carrier wavelength (equal to $2\pi /k_0$), by $L$ the typical propagation distance, and by $r_0$
the radius of the initial transverse source. The paraxial regime holds when the wavelength $\lambda_0$ is much smaller
than the radius $r_0$, and when the propagation distance is smaller than or of the order of $r_0^2/\lambda_0$ (the so-called Rayleigh length).
The white-noise paraxial regime that we address in this paper holds when, additionally, the medium has random fluctuations,
the typical amplitude of the medium fluctuations is small, and the correlation length of the medium fluctuations is larger than the wavelength and smaller than the propagation distance.
In this regime the solution of the time-harmonic wave equation (\ref{eq:helm}) can be approximated by \cite{garniers1}
$$
v(z,\bx) = \frac{i}{2k_0} u( z,\bx) \exp \big( i k_0 z\big) ,
$$
where $(u(z,\bx) )_{z \in [0,\infty),\bx\in \RR^2}$ is the solution 
of the  It\^o-Schr\"odinger equation
\begin{equation}
\label{eq:model}
 d {u}(z,\bx)   =     
          \frac{ i }{2k_0} \Delta_{\bx}   {u}(z,\bx) dz
   +   \frac{ik_0}{2}   {u} (z,\bx ) \circ  d{B}(z,\bx) 
  , 
\end{equation}
with the initial condition in the plane $z=0$:
$$
 {u}(z= 0,\bx )  = f(\bx) .
$$
Here the symbol $\circ$ stands for the Stratonovich stochastic integral and
  $B(z,\bx)$ is a real-valued  Brownian field over $[0,\infty) \times \RR^2$ with  covariance
 \begin{equation}
 \label{defB}
\EE[   {B}(z,\bx)  {B}(z',\bx') ] =  
 {\min\{z, z'\}}   {C}(\bx - \bx')   .
\end{equation}
The model (\ref{eq:model}) can be obtained from the 
scalar wave equation (\ref{eq:helm}) by a separation of scales technique in which the three-dimensional 
fluctuations of the index of refraction $n(z,\bx)$ are described by a zero-mean stationary random process $\nu(z,\bx)$ with 
mixing properties: $n^2(z,\bx)=1+\nu(z,\bx)$. The covariance function $C(\bx)$ in (\ref{defB}) is then given in terms
of the two-point statistics of the random process $\nu$ by
\begin{equation}
\label{corrdef}
 {C} (\bx) := \int_{-\infty}^\infty \EE[\nu(z'+  z,\bx'+  \bx) \nu(z',\bx')]   dz .  
\end{equation}
The covariance function $C$ is assumed to satisfy the following hypothesis:
\begin{equation}
\label{hyp:h}
\mbox{$C \in L^1(\RR^2)$ and $C(0) < \infty$.}
\end{equation}
The condition $C\in L^1(\RR^2)$ imposes that the Fourier transform
$\hat{C}$ is continuous and bounded by  Lebesgue dominated convergence theorem, 
and it is also nonnegative by Bochner's  theorem (it is the power spectral density of a stationary process).
The condition $C(0)< \infty$ then shows that $\hat{C} \in L^1(\RR^2)$,
and therefore $C$ is continuous and bounded.   \\
The white-noise paraxial model is widely used in the physical literature.
It simplifies the full wave equation (\ref{eq:helm}) by replacing it with the initial value-problem (\ref{eq:model}).
It was studied mathematically in \cite{dawson84}, in which the solution of
(\ref{eq:model})  is shown to be the solution of a martingale problem 
whose $L^2$-norm is preserved in the case $f \in L^2(\RR^2)$.
The derivation of the  It\^o-Schr\"odinger equation
(\ref{eq:model}) from the three-dimensional wave equation in 
randomly scattering medium is given in~\cite{garniers1}.

\section{Main Result and Quasi Gaussianity}
\label{sec:main}%
Modeling with the white-noise paraxial model is often motivated 
by  propagation through  randomly heterogeneous media.
The typical objective for such
modeling is to describe some communication or imaging scheme,
say with an object buried in the random medium. 
In many wave propagation and imaging  scenarii the quantity of interest
is   given by a quadratic quantity of the field $u$.   
For instance, in the time-reversal problems a wave field emitted by the source is
recorded on an array,  then time-reversed and re-propagated into the medium \cite{fink}. 
The forward and time-reversed propagation paths give rise to a quadratic
quantity in the field itself for the re-propagated field.  
 Moreover, in important  imaging approaches,
in particular passive imaging techniques \cite{GP}, 
the image is formed based on computing cross correlations of the field measured over an array
again giving a quadratic expression in  the field for the quantity of 
interest,  the correlations.   In  a number of situations, in particular in optics, 
 the measured quantity is
an  intensity, again a quadratic quantity in the field.
As we explain in  Section \ref{sec:second} the  expected value of 
such quadratic quantities 
can  in the white-noise paraxial regime  be computed explicitly. 
In imaging applications this allows to compute the mean image
and assess issues like resolution. However, it is important to go beyond this
description and calculate  the signal-to-noise ratio which requires
to compute a fourth-order moment of the wave field.
 Despite the importance of the signal-to-noise ratio hitherto no rigorous result has been available that accomplishes this task.  Indeed explicit expressions for the fourth moments has been a long standing open problem.  
This is what  we push through  in this paper.  In the context of design of imaging techniques this insight is important
to make proper balance in between noise and resolution in the image.  
We remark that in certain regimes one may be able to prove 
statistical stability, that is,  that the signal-to-noise ratio goes to infinity in the scaling limit
\cite{PRS04,PRS07}. The results we present here are more general in the sense that we can actually 
describe a finite signal-to-noise ratio and how  the parameters of the problem determine this. 

To summarize  and explicitly articulate the main result regarding the fourth-order  moment  we
consider first the first and second order moments  of $u$ in (\ref{eq:model})  
in the context when  $f(x)= \exp[-|\bx|^2/(2r_0^2)]$. 
We use the notations for the first and second-order moments
\begin{equation}
\label{def:mu1}
 \mu_1(z,\bx) :=\EE[ u(z,\bx)] , \quad \mu_2(z,\bx,\by) :=\EE[ u(z,\bx)   \overline{u(z,\by)} ] ,
\end{equation}
 Note that $\mu_2$ is given explicitly in (\ref{eq:transb12}). 
  For the second centered moment we use the notation:
\begin{equation}
\tilde\mu_2(z,\bx,\by) :=\mu_2(z,\bx,\by)- \mu_1(z,\bx)   \overline{\mu_1(z,\by)} .
\end{equation}

Then,  to obtain an expression for the fourth-order moment,  one heuristic approach 
often used in the literature \cite{fann09,ishimaru}  is to assume Gaussianity.
Consider any complex circularly symmetric Gaussian process $(Z(\bx))_\bx$  then we have 
\cite{reed62} that  the fourth-order moment can be expressed
in terms of the second-order moments by the Gaussian summation rule    as
\begin{eqnarray}
 \nonumber
 \EE \big[ Z(\bx_1) Z(\bx_2) \overline{Z(\by_1) Z(\by_2) } \big]  &=&
 \EE\big[ Z(\bx_1)  \overline{Z(\by_1)  }   \big] \EE \big[ Z(\bx_2) \overline{ Z(\by_2) } \big] 
\\
&&+
 \EE\big[ Z(\bx_1)  \overline{ Z(\by_2) }  \big] \EE \big[Z(\bx_2) \overline{ Z(\by_1) } \big]  .
\label{eq:gaussrule}
\end{eqnarray}
If the centered field $u(z,\bx) -\mu_1(z,\bx)$ were a  complex circularly symmetric Gaussian process,
then the fourth-order moment of the field $u$ defined by:
$$
\mu_4(z,\bx_1,\bx_2,\by_1,\by_2) :=  \EE[ u(z,\bx_1)   u(z,\bx_2)   \overline{u(z,\by_1) u(z,\by_2)} ]   
$$ 
would satisfy:
\begin{eqnarray*}
&&
\mu_4(z,\bx_1,\bx_2,\by_1,\by_2) =
     \mu_1(z,\bx_1)   \mu_1(z,\bx_2)  \overline{\mu_1(z,\by_1)} \overline{  \mu_1(z,\by_2)}  \\
    &&  \quad +  \mu_1(z,\bx_1)  \overline{ \mu_1(z,\by_1)}  \tilde\mu_2(z,\bx_2,\by_2)  
     + \mu_1(z,\bx_2)   \overline{\mu_1(z,\by_1)}  \tilde\mu_2(z,\bx_1,\by_2)   \\ 
     &&  \quad
     + \mu_1(z,\bx_1)  \overline{ \mu_1(z,\by_2) } \tilde\mu_2(z,\bx_2,\by_1) 
      + \mu_1(z,\bx_2)  \overline{ \mu_1(z,\by_2)}  \tilde\mu_2(z,\bx_1,\by_1) 
     \\
 &&   \quad+  \tilde\mu_2(z,\bx_1,\by_1) \tilde\mu_2(z,\bx_2,\by_2) +
 \tilde\mu_2(z,\bx_1,\by_2) \tilde\mu_2(z,\bx_2,\by_1) ,
\end{eqnarray*}
or equivalently:
\begin{eqnarray}
\nonumber
\mu_4(z,\bx_1,\bx_2,\by_1,\by_2) &=&\mu_2(z,\bx_1,\by_1)  \mu_2(z,\bx_2,\by_2) +
 \mu_2(z,\bx_1,\by_2)  \mu_2(z,\bx_2,\by_1) \\
&& -
    \mu_1(z,\bx_1)   \mu_1(z,\bx_2)  \overline{\mu_1(z,\by_1)} \overline{  \mu_1(z,\by_2)} .
\label{eq:gsr}
\end{eqnarray}
This result is not correct in general. 
For instance, in the spot-dancing regime addressed in \cite{dawson84,furutsu72,furutsu73},
the explicit calculation of the moments of all orders is carried out and exhibits non-Gaussian statistics,
in particular, the intensity follows a Rice-Nakagami statistics.
The spot-dancing regime is valid for a narrow initial beam,  strong medium fluctuations, and short propagation distance:
\begin{eqnarray*}
r_0 = r_0'\eps, \quad  C(\bx) = \eps^{-2} C'(\bx), \quad z  = z' \eps ,
\end{eqnarray*}
with $\eps \ll 1$ and the primed quantities of order  one. 

We show however in this paper that in the so-called
scintillation regime the Gaussian summation rule (\ref{eq:gsr}) is valid .
The scintillation regime is discussed  in detail in Section
\ref{sec:regime}, it is characterized by a wide initial beam, a long propagation distance, and weak
medium fluctuations:
\begin{eqnarray*}
r_0 = r_0'/\eps, \quad  C(\bx) = \eps C'(\bx), \quad z  = z'/\eps ,
\end{eqnarray*}
with $\eps \ll 1$ and 
 the primed quantities of order  one. 
Moreover, in the scintillation regime, if the source spatial profile is Gaussian with radius $r_0$:
\begin{equation}
\label{inigaus}
f (\bx) =
 \exp\Big( - \frac{|\bx|^2}{2r_0^2 }\Big)  ,
\end{equation}
then
\begin{eqnarray*}
\mu_2(z,\bx,\by) &=&
\frac{r_0^2}{4\pi}   \exp\left( - {\frac{k_0^2  C({\bf 0}) z}{4}}  \right)
 \int_{\RR^2}
\exp \Big(  -
\frac{r_0^2   |\bxi|^2 }{4} + i \frac{\bxi \cdot (\bx+\by)}{2}  \Big)\\
&& \times \exp \Big( \frac{k_0^2}{4} \int_0^z
 {C}\big( \bx-\by  - \bxi \frac{z'}{k_0}  \big)
  dz' \Big)  
 d \bxi 
 \end{eqnarray*}
and  
\begin{eqnarray*}
 \mu_1(z,\bx) =     \exp\left( - {\frac{ |\bx|^2}{2 r_0^2}}  \right) 
  \exp\left( - {\frac{k_0^2  C({\bf 0}) z}{8}}  \right) .
\end{eqnarray*}
Note that, in the scintillation regime, the field is partially coherent:
the coherent field (i.e., the mean field $\mu_1$) has an amplitude which is of the same order
as the standard deviation of the zero-mean incoherent field (i.e., the fluctuations of the field $u-\mu_1$).
The surprising result that we report in this paper is that the incoherent field behaves like a random field
with Gaussian statistics, as far as the fourth-order moments are concerned.

Finally in the strongly scattering scintillation regime when $k_0^2 {C}({\bf 0}) z \gg 1$ so that the mean field $\mu_1$ is vanishing and 
the field becomes completely incoherent, we   have in fact:
\begin{eqnarray*} 
  \mu_4(z,\bx_1,\bx_2,\by_1,\by_2)  \approx      \mu_2(z,\bx_1,\by_1)  \mu_2(z,\bx_2,\by_2) +
  \mu_2(z,\bx_1,\by_2)  \mu_2(z,\bx_2,\by_1)  .
\end{eqnarray*}
   
These results can now be used to discuss a wide range of applications 
in imaging and wave propagation.  
The  fourth moment  is a fundamental quantity in the context of waves in complex media and the above result
is the first rigorous derivation of it that makes explicit the  particular scaling regime in
which  it is valid, moreover, 
when in fact the Gaussian assumption can be used.  

In this paper we also discuss application
to characterization of the scintillation in Section \ref{sec:intensity}. The scintillation index 
describes the relative intensity fluctuations for the wave field. 
Despite being a fundamental physical quantity associated for instance with 
light propagation through the atmosphere, a rigorous derivation was not obtained before.
We moreover  give an explicit characterization of the signal to noise ratio for  the Wigner  transform
in Section \ref{sec:wigner}. The Wigner transform is a fundamental quadratic form of the field  
that is useful  in the context  of analysis of problems involving paraxial or Schr\"odinger equations, for instance time-reversal problems. 
 
 We remark finally that the results derived here can be useful in the analysis 
 of ghost imaging  experiments  \cite{cheng09,li10,shapiro12}, 
 enhanced  focusing \cite{popoff14,vellekoop10,vellekoop07,vellekoop08}  and super-resolution imaging problems \cite{katz12,mosk12,popoff10}, 
 and intensity correlation \cite{webb2,webb3}. Results on this will 
 be reported  elsewhere.
 
{\it Ghost imaging}  is a fascinating recent imaging methodology.
It can be interpreted as a correlation-based technique
since it gives an image of an object by correlating the intensities
measured by two detectors, a high-resolution detector  that does not view the object and a
low-resolution detector that does view the object.
The resolution of the image depends on the coherence properties
of the noise sources used to illuminate the object, 
and on the scattering properties of the medium. 
This problem can be understood at the mathematical level
 by using the results presented in this paper. 
  
{\it Enhanced focusing} refers to  schemes for communication and imaging 
in a  case where a reference signal propagating through the channel is 
available. Then this information can be used to design an optimal probe that focuses tightly 
at the desired focusing point. How to optimally design and analyze such schemes, 
given the limitations of the transducers and so on,  can be analyzed using the moment
theory presented in this paper.  More generally 
{\it Super resolution}  refers to the case where one tries to go beyond the classic diffraction limited resolution
in imaging systems.

{\it Intensity correlations} is a recently proposed scheme for communication in the optical regime
 that is based on using cross corrections of intensities, as measured in this regime, for communication.
 This is a promising scheme for communication through relatively strong clutter.  By using the 
 correlation of the intensity or speckle for different incoming angles of the source one can
 get spatial information about the  source. The idea of using the information about the statistical
 structure of speckle to enhance signaling is very interesting and corroborates the idea
 that modern schemes for communication and imaging   require a mathematical theory for
 analysis of high-order moments. 
     
  The results derived in this paper have already opened the mathematical
  analysis of important imaging problems and we believe that many more problems
  than those mentioned here will benefit from  the results regarding the fourth moments.
  In fact,   enhanced transducer technology and sampling schemes allow for
  using finer aspects of the  wave field involving second- and fourth-order moments 
  and in such complex cases a rigorous mathematical analysis is important to support,
  complement,  or actually disprove, statements based on physical intuition alone. 
 
\section{The Mean Field}
\label{sec:mom1}%
In this section we give the expression of the mean field, that is, the first-order moment 
$\mu_1(z,\bx)$ defined by (\ref{def:mu1}).
Using  It\^o's formula for Hilbert space-valued processes \cite{miyahara82}
(the process $u\big(z,\bx)$ takes values in $L^2(\RR^2)$), 
 we find that the function $\mu_1$ satisfies the damped
Schr\"odinger equation  in $L^2(\RR^2)$:
\begin{equation}
\frac{\partial \mu_1}{\partial z} =       
          \frac{ i }{2k_0} \Delta_{\bx}  \mu_1
  -\frac{k_0^2 C({\bf 0})}{8}  \mu_1
  , 
\end{equation}
with the initial condition $\mu_1(z= 0,\bx )  = f(\bx)$.
This equation can be solved in the Fourier domain which gives 
\begin{equation}
\mu_1(z,\bx)       =    \frac{1}{(2\pi)^2} \int \hat{f}(\bxi) \exp \Big( i \bxi \cdot \bx - \frac{i |\bxi|^2 z}{2k_0} \Big) d\bxi 
 \exp \Big( -\frac{k_0^2 C({\bf 0}) z}{8} \Big),
\end{equation}
where $\hat{f}$ is the Fourier transform of the initial field:
\begin{equation}
\label{def:fourierf}
\hat{f} (\bxi) = \int {f} (\bx) \exp ( - i \bxi \cdot \bx) d\bx .
\end{equation}
In this paper, unless mentioned explicitly, all integrals are over $\RR^2$.
The exponential damping of the mean field is noticeable,
it can be physically explained by the random phase that the wave acquires as it propagates
through the random medium.
If the initial condition is the Gaussian profile (\ref{inigaus}) then we get
\begin{equation}
\mu_1(z,\bx)       =    
  \frac{r_0^2}{r_z^2} \exp \Big( - \frac{k_0^2 C({\bf 0}) z}{8}\Big)
   \exp \Big( -\frac{|\bx|^2}{2 r_z^2 } \Big),\quad \quad  r_z^2 = r_0^2 \left( 1 + \frac{i z}{k_0 r_0^2}  \right)   .
   \label{eq:expresmu1}
\end{equation}

\section{The General Moment Equations}
\label{sec:mom}%
The main tool for describing wave statistics are the finite-order moments.
We show in this section that in the context of the It\^o-Schr\"odinger equation (\ref{eq:model})
the moments of the field satisfy a closed system at each order
\cite{fps,ishimaru}.
For $p\in \NN$,
we define
\begin{equation}
\label{def:generalmoment}
\mu_{2p}\big(z , (\bx_j)_{j=1}^{p} ,  (\by_l)_{l=1}^{p} \big) :=
\EE \Big[ \prod_{j=1}^p {u}(z,\bx_j) 
 \prod_{l=1}^p \overline{{u}(z,\by_l)} \Big], 
\end{equation}
for $\bx_j  , \by_l  \in {\RR^2}  $ for $j,l=1,\ldots,p$.
Note that here the number of conjugated terms equals the number of non-conjugated 
terms, otherwise the moments decay relatively rapidly to zero due to unmatched
random phase terms associated with random travel time perturbations, as seen in the previous section.
Using the stochastic equation (\ref{eq:model}) 
and It\^o's formula for Hilbert space-valued processes \cite{miyahara82}
(the process $u\big(z,\bx_1)u\big(z,\bx_2)\overline{u\big(z,\by_1)u\big(z,\by_2)}$ takes values in $L^2(\RR^2 \times \cdots \times \RR^2)$),
 we find that the function $\mu_{2p}$ satisfies the 
Schr\"odinger-type system in $L^2(\RR^2 \times \cdots \times \RR^2)$:
\begin{eqnarray}
&&\frac{\partial \mu_{2p}}{\partial z} = \frac{i}{2k_0}  \Big( \sum_{j=1}^p \Delta_{\bx_j}
- \sum_{l=1}^p \Delta_{\by_l} \Big) \mu_{2p} + \frac{k_0^2}{4} {U}_{2p} \big( (\bx_j)_{j=1}^p, (\by_l)_{l=1}^{p}\big)
 \mu_{2p}  ,\\
&& \mu_{2p}(z=0) = \prod_{j=1}^p f(\bx_j) 
 \prod_{l=1}^p \overline{f(\by_l)}  ,
\end{eqnarray}
with the generalized potential
\begin{eqnarray}
\nonumber
&&{U}_{2p}\big( (\bx_j)_{j=1}^p, (\by_l)_{l=1}^{p}\big)  \\
\nonumber
&& := 
\sum_{j,l=1}^p {C}(\bx_j-\by_l) 
- \frac{1}{2}
\sum_{j,j'=1}^p {C}( \bx_j-\bx_{j'})
- \frac{1}{2}
\sum_{l,l'=1}^p {C}( \by_l-\by_{l'}) \\
&&
=
\sum_{j,l=1}^p {C}(\bx_j-\by_l) 
-\sum_{1 \leq j<j'\leq p} {C}( \bx_j-\bx_{j'})
-\sum_{1 \leq l<l'\leq p} {C}( \by_l-\by_{l'}) -
p{C}({\bf 0}) \, . 
\hspace*{0.3in}
\end{eqnarray}

We introduce the Fourier transform
\begin{eqnarray}
\nonumber
&&\hat{\mu}_{2p}\big(z, (\bxi_j)_{j=1}^p, (\bzeta_l)_{l=1}^p \big) =
\iint \mu_{2p}\big(z, (\bx_j)_{j=1}^{p}  , (\by_l)_{l=1}^{p}  \big) \\
&& \quad\quad \times \exp  \Big( - i \sum_{j=1}^p \bx_j \cdot \bxi_j + i \sum_{l=1}^p \by_l \cdot \bzeta_l \Big)
d\bx_1\cdots d\bx_p d\by_1\cdots d\by_p  .
\end{eqnarray}
It satisfies
\begin{eqnarray}
\label{eq:fouriermom}
&&\frac{\partial \hat{\mu}_{2p}}{\partial z} = - \frac{i}{2k_0}  \Big( \sum_{j=1}^p |\bxi_j|^2 
- \sum_{l=1}^p |\bzeta_l |^2 \Big) \hat{\mu}_{2p} +  \frac{k_0^2}{4}  \hat{\cal U}_{2p} \hat{\mu}_{2p}  , \\
&& \hat{\mu}_{2p}(z=0) = \prod_{j=1}^p \hat{f}(\bxi_j) 
 \prod_{l=1}^p \overline{ \hat{f}(\bzeta_l)}  ,
\end{eqnarray}
where $\hat{f}$ is the Fourier transform (\ref{def:fourierf}) of the initial field
and the operator $\hat{\cal U}_{2p}$ is defined by
\begin{eqnarray}
\nonumber
&&\hat{\cal U}_{2p} \hat{\mu}_{2p} = \frac{1}{(2\pi)^2}
\int \hat{C}(\bk) \Big[
\sum_{j,l=1}^p   \hat{\mu}_{2p}(\bxi_j-\bk,\bzeta_l-\bk)\\
&&-\sum_{1 \leq j<j'\leq p} \hat{\mu}_{2p}(\bxi_j-\bk,\bxi_{j'}+\bk) 
 -\sum_{1 \leq l<l'\leq p} \hat{\mu}_{2p}(\bzeta_l-\bk,\bzeta_{l'}+\bk) 
 - p \hat{\mu}_{2p} \Big]
 d\bk , \hspace*{0.25in}
\end{eqnarray} 
where we only write the arguments that are shifted.
It turns out that the equation for the Fourier transform $\hat{\mu}_{2p}$ is easier to solve 
than the one for $\mu_{2p}$ as we will see below.

\section{The Second-Order Moments}
\label{sec:second}%
The second-order moments play an important role, as they give the mean intensity profile 
and the correlation radius of the transmitted beam
\cite{feizulin,garniers2},
they can be used to analyze time reversal experiments  \cite{blomgren,PRS04} and 
wave imaging problems \cite{dehoopsolna09,dehoopgarnier13}, and we will need them to compute the scintillation index
of the transmitted beam and the variance of the Wigner transform.
We describe them in detail in this section.

\subsection{The Mean Wigner Transform}
\label{sec:Wtransform}
The second-order moments (\ref{def:mu1}) satisfy the system $L^2(\RR^2 \times  \RR^2)$:
\begin{equation}
\label{eq:mom11}
\frac{\partial {\mu}_2 }{\partial z} = \frac{i}{2k_0}  \big(  \Delta_{\bx}
-  \Delta_{\by} \big) {\mu}_2  + \frac{k_0^2}{4} \big( C(\bx-\by)-C({\bf 0})  \big) {\mu}_2    ,
\end{equation}
starting from $ {\mu}_2 (z=0,\bx,\by) = f(\bx)   \overline{f(\by)} $.
The second-order moment is related to the mean Wigner transform defined by 
\begin{equation}
{W}_{\rm m} ( z,\br,\bxi) := 
\int 
\exp \big( - i  \bxi \cdot \bq  \big)
\EE \left[ {u}\big(z,  \br+\frac{\bq}{2} \big)
 \overline{u} \big( z,   \br-\frac{\bq}{2}\big) \right]  d \bq  ,
\end{equation}
that is the angularly-resolved  mean wave energy density.
The mean Wigner transform $W_{\rm m}(z,\cdot)$ is in $L^2(\RR^2 \times \RR^2)$ and 
$\| W_{\rm m}(z,\cdot) \|_{L^2 (\RR^2\times \RR^2)} \leq 2 \pi \| f \|_{L^2(\RR^2)}^2$.
It is also bounded by $\| W_{\rm m}(z,\cdot) \|_{L^\infty (\RR^2\times \RR^2)} \leq 4 \| f \|_{L^2(\RR^2)}^2$.
Using (\ref{eq:mom11}) we find that it satisfies the closed system
\begin{equation}
\label{systemWT2rapid}
\frac{\partial {W}_{\rm m} }{\partial z} + \frac{1}{k_0}
{ \bxi}\cdot \nabla_{\br} W_{\rm m}
=\frac{k_0^2}{4 (2\pi)^2}
\int \hat{C}( \bk) \Big[ 
W_{\rm m} (   \bxi - \bk   ) 
 - W_{\rm m}(  \bxi ) 
\Big]  d \bk ,
\end{equation}
starting from ${W}_{\rm m}( z= 0 ,\br,\bxi) =W_0(\br,\bxi)$, which is the Wigner transform of the initial field~$f$:
$$
W_0(\br,\bxi) := \int \exp \big(- i\bxi \cdot \bq \big) f \big( \br + \frac{\bq}{2}\big) \overline{f} \big( \br - \frac{\bq}{2}\big)  d\bq.
$$
Eq.~(\ref{systemWT2rapid}) has the form of a radiative 
transport equation for the wave energy density $W_{\rm m}$.  In this context 
$k_0^2{C}({\bf 0})/4$ is the total
scattering  cross-section and $k_0^2 \hat{C}(\cdot)/[4(2\pi)^2]$  is
the  differential scattering cross-section that gives the mode conversion rate.

By taking a Fourier transform in $\br$ and an inverse Fourier transform in $\bxi$  of Eq.~(\ref{systemWT2rapid}):
$$
\hat{W}_{\rm m} ( z,\bzeta,\bq) = \frac{1}{(2\pi)^2}  
\iint 
\exp \big( - i  \bzeta \cdot \br + i \bxi \cdot \bq  \big)
{W}_{\rm m} ( z,\br,\bxi) d \bxi d\br  ,
$$
we obtain a transport equation:
$$
\frac{\partial \hat{W}_{\rm m} }{\partial z} + \frac{1}{k_0}
{ \bzeta}\cdot \nabla_{\bq} \hat{W}_{\rm m}
=\frac{k_0^2}{4}
\big[ C(\bq)-C({\bf 0})\big] \hat{W}_{\rm m},
$$
that can be solved and we find
 the following integral representation for $W_{\rm m}$:
\begin{eqnarray}
\nonumber
W_{\rm m}(z,\br,\bxi)&=&  \frac{1}{(2\pi)^{2}} \iint
\exp \Big( i  \bzeta \cdot \big( \br -\bxi \frac{z}{k_0} \big) - i  \bxi  \cdot \bq \Big)\hat{W}_0\big(\bzeta,\bq  \big) 
\\
&& \times  \exp \Big( \frac{k_0^2}{4} \int_0^z {C} \big( \bq + \bzeta \frac{z'}{k_0} \big)-{C}(
 {\bf 0}) d z' \Big)
 d\bzeta d \bq ,
\label{solWT1b}
\end{eqnarray}
where $\hat{W}_0$ is defined in terms of  the initial field 
$f$ as:
\begin{equation}
\label{eq:part}
\hat{W}_0(\bzeta,\bq)  = \int 
\exp \big( - i  \bzeta \cdot \br  \big)
  f\big(   \br+\frac{\bq}{2} \big)
 \overline{f} \big(     \br-\frac{\bq}{2}\big)    d \br  .
\end{equation}

\subsection{The Mutual Coherence Function}
The mutual coherence function is defined by:
\begin{equation}
\Gamma^{(2)}(z,\br,\bq) := \mu_2\Big(z, \br+\frac{\bq}{2} , \br-\frac{\bq}{2} \Big) =  \EE \left[ {u}\big( z, \br+\frac{\bq}{2} \big) 
\overline{u} \big(z,\br-\frac{\bq}{2} \big) \right]  ,
\end{equation}
where $\br$ is the mid-point and $\bq$ is the offset.
It can be computed  by taking the inverse Fourier transform of the expression (\ref{solWT1b}):
\begin{eqnarray}
\nonumber
\Gamma^{(2)}(z,\br,\bq) &=& \frac{1}{(2\pi)^2}
\int \exp \big( i \bxi \cdot \bq \big) W_{\rm m}(z,\br,\bxi) d\bxi \\
\nonumber
&=&
 \frac{1}{(2\pi)^{2}} \int
\exp \big( i  \bzeta \cdot \br  \big) 
\hat{W}_0\big(\bzeta,\bq -\bzeta\frac{z}{k_0} \big) \\
&&\times  \exp \Big( \frac{k_0^2}{4} \int_0^z {C} \big( \bq - \bzeta \frac{z'}{k_0} \big)-{C}(
 {\bf 0}) d z' \Big)
 d\bzeta   .
 \label{expres:gamma2}
\end{eqnarray}
Let us examine the particular initial condition (\ref{inigaus})
which corresponds to a Gaussian-beam wave.
If the initial condition is the Gaussian profile (\ref{inigaus}),
then we have
\begin{equation}
\hat{W}_0(\bzeta,\bq) = \pi r_0^2 \exp\Big( -\frac{r_0^2 |\bzeta|^2}{4}  - \frac{|\bq|^2}{4 r_0^2 }\Big)  ,
\end{equation} 
and we find from (\ref{expres:gamma2}) that the mutual coherence function has the form
\begin{eqnarray}
\nonumber
\Gamma^{(2)}(z,\br,\bq) &=&
\frac{r_0^2 }{4 \pi }  \int
\exp \Big(-\frac{1}{4 r_0^2} \big|  \bq - \bzeta  \frac{z}{k_0}\big|^2 -
\frac{r_0^2   |\bzeta|^2 }{4} + i \bzeta \cdot \br \Big)\\
&& \times \exp \Big( \frac{k_0^2}{4} \int_0^z
 {C}\big( \bq - \bzeta \frac{z'}{k_0}  \big)
- { C}( {\bf 0} ) dz' \Big)
 d \bzeta .
\label{eq:transb12}
\end{eqnarray}

\section{The Fourth-Order Moments} 
\label{sec:fourth}%
We consider the fourth-order moment ${\mu}_4$ of the field,
which is the main quantity of interest in this paper,
and parameterize  the four points 
$\bx_1,\bx_2,\by_1,\by_2$ in (\ref{def:generalmoment}) in the special way:
\begin{eqnarray}
\label{eq:reliexr1}
\bx_1 = \frac{\br_1+\br_2+\bq_1+\bq_2}{2}, \quad \quad 
\by_1 = \frac{\br_1+\br_2-\bq_1-\bq_2}{2}, \\
\bx_2 = \frac{\br_1-\br_2+\bq_1-\bq_2}{2}, \quad \quad 
\by_2 = \frac{\br_1-\br_2-\bq_1+\bq_2}{2}.
\label{eq:reliexr2}
\end{eqnarray}
In particular $\br_1/2$ is the barycenter of the four points $\bx_1,\bx_2,\by_1,\by_2$:
\begin{eqnarray*}
\br_1 = \frac{\bx_1+\bx_2+\by_1+\by_2}{2} , \quad \quad 
\bq_1 = \frac{\bx_1+\bx_2-\by_1-\by_2}{2}, \\
\br_2 = \frac{\bx_1-\bx_2+\by_1-\by_2}{2}, \quad \quad 
\bq_2 = \frac{\bx_1-\bx_2-\by_1+\by_2}{2}.
\end{eqnarray*}

We denote by $\mu$ the fourth-order moment in these new variables:
\begin{equation}
\mu(z,\bq_1,\bq_2,\br_1,\br_2) := 
\mu_4  (z, 
\bx_1  ,
\bx_2 ,
\by_1 ,  
\by_2 
)
\end{equation}
with $\bx_1,\bx_2,\by_1,\by_2$ given by (\ref{eq:reliexr1}-\ref{eq:reliexr2}) in terms of $\bq_1,\bq_2,\br_1,\br_2$.

In the variables $(\bq_1,\bq_2,\br_1,\br_2)$ the function ${\mu}(z,\cdot)$ satisfies the system in
$L^2(\RR^2 \times\RR^2 \times \RR^2 \times \RR^2)$:
\begin{equation}
\label{eq:M20}
\frac{\partial {\mu}}{\partial z} = \frac{i}{k_0} \big( \nabla_{\br_1}\cdot \nabla_{\bq_1}
+
 \nabla_{\br_2}\cdot \nabla_{\bq_2}
\big)  {\mu} + \frac{k_0^2}{4} {U}(\bq_1,\bq_2,\br_1,\br_2) {\mu}   ,
\end{equation}
with the generalized potential
\begin{eqnarray}
\nonumber
{U}(\bq_1,\bq_2,\br_1,\br_2) &:=& 
{C}(\bq_2+\bq_1) 
+
{C}(\bq_2-\bq_1) 
+
{C}(\br_2+\bq_1) 
+
{C}(\br_2-\bq_1) \\
&&- 
{C}( \bq_2+\br_2) - {C}( \bq_2-\br_2) - 2 {C}({\bf 0})  .
\end{eqnarray}
Note in particular that the generalized potential does not depend on the barycenter $\br_1$,
and this comes from the fact that the medium
is statistically homogeneous.
If we assume that the source spatial profile is the Gaussian (\ref{inigaus}) with radius $r_0$,
then the initial condition for Eq.~(\ref{eq:M20}) is
$$
{\mu}(z=0,\bq_1,\bq_2,\br_1,\br_2) = \exp \Big( - \frac{ |\bq_1|^2+ |\bq_2|^2 + |\br_1|^2+ |\br_2|^2}{2r_0^2} \Big).
$$

The Fourier transform (in $\bq_1$, $\bq_2$, $\br_1$, and $\br_2$) of the fourth-order moment
is defined by:
\begin{eqnarray}
\nonumber
\hat{\mu}(z,\bxi_1,\bxi_2,\bzeta_1,\bzeta_2) 
&=& 
\iint {\mu}(z,\bq_1,\bq_2,\br_1,\br_2)  \\
&& \hspace*{-0.8in}
\times
\exp  \big(- i\bq_1 \cdot \bxi_1- i\bq_2 \cdot \bxi_2- i\br_1\cdot \bzeta_1- i\br_2\cdot \bzeta_2\big) d\bq_1d\bq_2 
d\br_1d\br_2 \label{eq:fourier} 
. \hspace*{0.3in} 
\end{eqnarray}
It satisfies
\begin{eqnarray}
\nonumber
&& 
\frac{\partial \hat{\mu}}{\partial z} + \frac{i}{k_0} \big( \bxi_1\cdot \bzeta_1+   \bxi_2\cdot \bzeta_2\big) \hat{\mu}
=
\frac{k_0^2}{4 (2\pi)^2} 
\int \hat{C}(\bk) \bigg[  
 \hat{\mu} (  \bxi_1-\bk, \bxi_2-\bk, \bzeta_1, \bzeta_2)  \\
\nonumber
&& \quad  + 
 \hat{\mu} (  \bxi_1-\bk,\bxi_2,  \bzeta_1, \bzeta_2-\bk)    
 +
 \hat{\mu} (  \bxi_1+\bk, \bxi_2-\bk, \bzeta_1, \bzeta_2)   \\
\nonumber
&&  \quad 
+ 
 \hat{\mu} (  \bxi_1+\bk,\bxi_2, \bzeta_1,  \bzeta_2-\bk)    -
2 \hat{\mu}(\bxi_1,\bxi_2, \bzeta_1, \bzeta_2) \\
&&  \quad 
-
 \hat{\mu} (  \bxi_1,\bxi_2-\bk, \bzeta_1, \bzeta_2-\bk)  
- \hat{\mu} (  \bxi_1,\bxi_2+\bk,  \bzeta_1, \bzeta_2-\bk) 
\bigg]d \bk ,
\label{eq:fouriermom0}
\end{eqnarray}
starting from $\hat{\mu}(z=0,\bxi_1,\bxi_2,\bzeta_1,\bzeta_2) = (2\pi r_0^2)^4 \exp ( - r_0^2 ( |\bxi_1|^2+ |\bxi_2|^2 + |\bzeta_1|^2+ |\bzeta_2|^2)/ 2)$.
The modified function $\tilde{\mu}$ defined by 
\begin{equation}
\label{def:MtildeMeps}
\tilde{\mu} (z,\bxi_1,\bxi_2 , \bzeta_1,\bzeta_2 \big)   =  \hat{\mu} \big(z,\bxi_1,\bxi_2, \bzeta_1,\bzeta_2 \big) 
 \exp \Big( \frac{i z}{k_0} (\bxi_2 \cdot \bzeta_2  +   \bxi_1 \cdot \bzeta_1) \Big)
\end{equation}
therefore satisfies:
\begin{equation}
\label{eq:tildemu}
\frac{\partial \tilde{\mu}}{\partial z}=
{\cal L}_z \tilde{\mu}(z)
\end{equation}
starting from 
\begin{equation}
\label{eq:icmu}
\tilde{\mu} (z=0,\bxi_1,\bxi_2,\bzeta_1,\bzeta_2) = (2\pi r_0^2)^4 \exp \Big( - r_0^2 \frac{ |\bxi_1|^2+ |\bxi_2|^2 + |\bzeta_1|^2+ |\bzeta_2|^2}{ 2} \Big)
,
\end{equation}
where the operator ${\cal L}_z$ from $ L^2(\RR^2 \times\RR^2 \times\RR^2 \times\RR^2 )$ into itself is defined by: 
\begin{eqnarray}
{\cal L}_z \psi(\bxi_1,\bxi_2,\bzeta_1,\bzeta_2)  &:=&
\frac{k_0^2}{4 (2\pi)^2} 
\int \hat{C}(\bk) \bigg[  - 2 \psi (  \bxi_1 ,\bxi_2, \bzeta_1,\bzeta_2)   \\
\nonumber
&&+
\psi (  \bxi_1-\bk,\bxi_2-\bk,  \bzeta_1,\bzeta_2) 
e^{i\frac{z}{k_0} \bk \cdot  (\bzeta_2 + \bzeta_1)} \\
&& \nonumber
+ 
\psi  (  \bxi_1-\bk,\bxi_2,  \bzeta_1,\bzeta_2-\bk) 
e^{i\frac{z}{k_0} \bk \cdot ( \bxi_2 +  \bzeta_1)}\\
\nonumber
&& +\psi  (  \bxi_1+\bk,\bxi_2-\bk,  \bzeta_1,\bzeta_2) 
e^{i\frac{z}{k_0} \bk \cdot (\bzeta_2 -   \bzeta_1)}\\
&& \nonumber
+ 
\psi (  \bxi_1+\bk,\bxi_2,  \bzeta_1,\bzeta_2-\bk) 
e^{i\frac{z}{k_0} \bk \cdot (\bxi_2 -  \bzeta_1)}\\
\nonumber
&&
-
\psi  (  \bxi_1,\bxi_2-\bk, \bzeta_1, \bzeta_2-\bk) 
 e^{i\frac{z}{ k_0} (  \bk \cdot (\bzeta_2+\bxi_2)-|\bk|^2 )} \\
 &&
- \psi ( \bxi_1,\bxi_2-\bk,  \bzeta_1,\bzeta_2+\bk) 
e^{i\frac{z}{ k_0} ( \bk \cdot (\bzeta_2-\bxi_2)+|\bk|^2)}
\bigg] d \bk . \hspace*{0.2in}
\label{def:calL}
\end{eqnarray}
$\tilde{\mu}(z,\cdot)$ is in $L^2(\RR^2 \times \RR^2 \times \RR^2 \times \RR^2 )$.
The following lemma applied with $p=2$ shows that it is the unique solution to (\ref{eq:tildemu}) with the initial
condition (\ref{eq:icmu}).
\begin{lemma}
\label{lem:normLz}%
Let $p \in [1,\infty]$.
For any $z$, the operator ${\cal L}_z$ is bounded from $L^p(\RR^2 \times \RR^2 \times \RR^2 \times \RR^2 )$ into itself and 
$
\| {\cal L}_z \|_{L^p \to L^p} \leq 2k_0^2 C({\bf 0)} 
$
uniformly in $z$.
\end{lemma}
\begin{proof}
Since $\hat{C}$ is non-negative we have for any $\psi  \in L^p(\RR^2 \times \RR^2 \times \RR^2 \times \RR^2 )$
by the Minkowski's  integral  inequality:
\begin{eqnarray*}
\hspace*{0.7in}
 \|   {\cal L}_z \psi   \|_{L^p}    \leq    \frac{2 k_0^2}{(2\pi)^2}  
   \int  d\bk  \hat{C}(\bk)  \| \psi \|_{L^p} 
=      2 k_0^2 C({\bf 0})   \|  \psi    \|_{L^p} .
\hspace*{0.7in}
\Box 
\end{eqnarray*}
\end{proof}
Note that initial condition (\ref{eq:icmu}) is not only in $L^2(\RR^2 \times \RR^2 \times \RR^2 \times \RR^2 )$
but also  in $L^1(\RR^2 \times \RR^2 \times \RR^2 \times \RR^2 )$.
Since ${\cal L}_z$ is bounded as a linear operator from $L^1$ to $L^1$, this shows that
$\tilde{\mu}(z,\cdot)$ and therefore $\hat{\mu}(z,\cdot)$ is in $L^1$.  
Since $\mu(z,\cdot)$ is the inverse Fourier transform of $\hat{\mu}(z,\cdot)$, this shows that 
$ \mu(z,\cdot)$ is continuous and bounded.

The resolution of the equation (\ref{eq:tildemu}) would give the expression of the fourth-order moment.
However, in contrast to the second-order moment, we cannot solve this equation and find a closed-form
expression of the fourth-order moment in the general case.
Therefore we address in the next sections a particular regime in which explicit expressions can be obtained.

\section{The Scintillation Regime and Main Result}
\label{sec:regime}%
In this paper we  address a regime which can be considered
as a particular case of the paraxial white-noise regime:  the scintillation regime. 
In \cite{garniers3} we  addressed this regime in the limit case of an infinite beam radius, that is, a plane wave. 
It turns out that the equation that characterizes the fourth-order moments can then be reduced to an equation in
$\RR^{2}\times\RR^{2}$ that can be solved.
There is no such simplification with an initial condition in the form of a beam.
Here we address the propagation of a beam with finite radius $r_0$  
and the equation to be studied is in $\RR^{2}\times\RR^{2}\times\RR^{2}\times\RR^{2} $.

In Appendix \ref{app:a} we explain the conditions for  validity of the scintillation regime
in the context of the wave equation (\ref{eq:helm}).
More directly, if we start from the  It\^o-Schr\"odinger equation (\ref{eq:model}),
then the scintillation regime is valid if the (transverse) correlation length of the Brownian field 
is smaller than the beam radius,
the standard deviation of the Brownian field is small, and the propagation distance is large.
If the correlation length is our reference length, this means that 
in this regime the covariance function ${C}^\eps$ is of the form:
\begin{equation}
\label{sca:sci}
C^\eps(\bx)= \eps C (\bx) , 
\end{equation}
the beam radius is of order $1/\eps$, i.e.  the initial source is of the form 
\begin{equation}
\label{def:feps}
f^\eps(\bx) = \exp \Big( - \frac{\eps^2 |\bx|^2}{2r_0^2} \Big),
\end{equation}
and the propagation distance is of order of $1/\eps$.
Here $\eps$ is a small dimensionless parameter and we will study the limit $\eps \to 0$.
Note that for simplicity we assume that the initial beam profile is Gaussian, which allows us to get 
closed-form expressions,  but the results could be extended to more general beam profiles.

Let us denote the rescaled function
\begin{equation}
\label{eq:renormhatM2}
\tilde{\mu}^\eps (z ,\bxi_1,\bxi_2,\bzeta_1,\bzeta_2) := \tilde{\mu} \Big(\frac{z}{\eps} ,\bxi_1,\bxi_2,\bzeta_1,\bzeta_2 \Big) .
\end{equation}

In the scintillation regime the rescaled function $\tilde{\mu}^\eps$ 
satisfies the equation with fast phases
\begin{equation}
\label{eq:tildemueps}
 \frac{\partial \tilde{\mu}^\eps}{\partial z}  ={\cal L}^\eps_z \tilde{\mu}^\eps, 
 \end{equation}
 where 
\begin{eqnarray}
\nonumber{\cal L}^\eps_z  \psi  (  \bxi_1 ,\bxi_2, \bzeta_1,\bzeta_2)&:=&
\frac{k_0^2}{4 (2\pi)^2} 
\int \hat{C}(\bk) \bigg[  - 2\psi (  \bxi_1 ,\bxi_2, \bzeta_1,\bzeta_2)   \\
\nonumber
&&+
\psi   (  \bxi_1-\bk,\bxi_2-\bk,  \bzeta_1,\bzeta_2) 
e^{i\frac{z}{\eps k_0} \bk \cdot  (\bzeta_2 + \bzeta_1)} \\
&& \nonumber
+ 
\psi  (  \bxi_1-\bk,\bxi_2,  \bzeta_1,\bzeta_2-\bk) 
e^{i\frac{z}{\eps k_0} \bk \cdot ( \bxi_2 +  \bzeta_1)}\\
\nonumber
&& + \psi  (  \bxi_1+\bk,\bxi_2-\bk,  \bzeta_1,\bzeta_2) 
e^{i\frac{z}{\eps k_0} \bk \cdot (\bzeta_2 -   \bzeta_1)}\\
&& \nonumber
+ 
\psi   (  \bxi_1+\bk,\bxi_2,  \bzeta_1,\bzeta_2-\bk) 
e^{i\frac{z}{\eps k_0} \bk \cdot (\bxi_2 -  \bzeta_1)}\\
\nonumber
&&
-
\psi (  \bxi_1,\bxi_2-\bk, \bzeta_1, \bzeta_2-\bk) 
 e^{i\frac{z}{\eps k_0} (  \bk \cdot (\bzeta_2+\bxi_2)-|\bk|^2 )} \\
 &&
- \psi  ( \bxi_1,\bxi_2-\bk,  \bzeta_1,\bzeta_2+\bk) 
e^{i\frac{z}{\eps k_0} ( \bk \cdot (\bzeta_2-\bxi_2)+|\bk|^2)}
\bigg] d \bk , \hspace*{0.2in}
\label{eq:tildeNeps}
\end{eqnarray}
and the initial condition 
(corresponding to (\ref{def:feps})) is
\begin{equation}
\label{eq:initialtildeM2eps}
\tilde{\mu}^\eps(z=0,\bxi_1,\bxi_2, \bzeta_1, \bzeta_2 ) = (2\pi)^8 \phi^\eps ( \bxi_1 )
\phi^\eps ( \bxi_2 )
\phi^\eps ( \bzeta_1 )\phi^\eps ( \bzeta_2 ) ,
\end{equation}
where we have denoted
\begin{equation}
\phi^\eps(\bxi) := \frac{r_0^2}{2\pi \eps^2} \exp \Big( -\frac{r_0^2}{2 \eps^2} |\bxi|^2\Big) .
\end{equation}
Note that $\phi^\eps$ belongs to $L^1$ and has a $L^1$-norm equal to one.
The asymptotic behavior as $\eps \to 0$ of the moments is therefore determined
by the solutions of partial differential equations with rapid phase terms.
A key limit theorem will allow us to get a representation of the fourth-order moments 
in the asymptotic regime $\eps \to 0$.
We will see that, although the initial condition (\ref{eq:initialtildeM2eps}) is concentrated in the four variables around 
an $\eps$-neighborhood of ${\bf 0}$, the evolution equation will spread it, except in the $\bzeta_1$-variable
which is a frozen parameter in the evolution equation (\ref{eq:tildemueps}). This is related to the fact that 
the generalized potential does not depend on $\br_1$ as the medium
is statistically homogeneous.
It  corresponds to the fourth-order moment not varying rapidly 
with respect to the spatial center coordinate $\br_1$ while in the other 
barycentric coordinates we have in general rapid variations   
induced by the medium fluctuations on this  scale.

Our goal is now to study the asymptotic behavior of $\tilde{\mu}^\eps$ as $\eps \to 0$.
We have the following result, which shows that $\tilde{\mu}^\eps$ exhibits a multi-scale behavior
as $\eps \to 0$, with some components evolving at the scale $\eps$ and 
some components evolving at the order one scale.
\begin{proposition}
\label{prop:sci1}%
Under Hypothesis (\ref{hyp:h}),
the function $\tilde{\mu}^\eps(z,\bxi_1,\bxi_2, \bzeta_1,\bzeta_2 ) $ can be expanded as
\begin{eqnarray}
\nonumber
&&
 \tilde{\mu}^\eps(z,\bxi_1,\bxi_2,  \bzeta_1,\bzeta_2 )  =
K(z)
\phi^\eps ( \bxi_1 )
\phi^\eps ( \bxi_2 )
\phi^\eps ( \bzeta_1 )
\phi^\eps ( \bzeta_2 ) \\
\nonumber
&& 
\quad
+
K(z) 
\phi^\eps \big( \frac{\bxi_1-\bxi_2}{\sqrt{2}}\big)
\phi^\eps ( \bzeta_1 )
\phi^\eps ( \bzeta_2 )
A\big(z, \frac{\bxi_2+\bxi_1}{2} ,\frac{\bzeta_2 + \bzeta_1}{\eps}\big) \\
\nonumber
&&  
\quad
+
K(z) 
\phi^\eps \big( \frac{\bxi_1+\bxi_2}{\sqrt{2}}\big)
\phi^\eps ( \bzeta_1 )
\phi^\eps ( \bzeta_2 )
A \big(z, \frac{\bxi_2-\bxi_1}{2} ,\frac{\bzeta_2- \bzeta_1}{\eps} \big) \\
\nonumber
&&  
\quad
+
K(z) 
\phi^\eps \big( \frac{\bxi_1-\bzeta_2}{\sqrt{2}}\big)
\phi^\eps ( \bzeta_1 ) \phi^\eps ( \bxi_2 )
A\big(z, \frac{\bzeta_2+\bxi_1}{2} ,\frac{\bxi_2+ \bzeta_1}{\eps}  \big) \\
\nonumber
&& 
\quad
+
K(z) 
\phi^\eps \big( \frac{\bxi_1+\bzeta_2}{\sqrt{2}}\big)
\phi^\eps ( \bzeta_1 ) \phi^\eps ( \bxi_2 )
A\big(z, \frac{\bzeta_2-\bxi_1}{2} ,\frac{\bxi_2- \bzeta_1}{\eps}  \big) \\
\nonumber
&& 
\quad
+K(z)\phi^\eps ( \bzeta_1 )\phi^\eps (\bzeta_2) 
A \big( z, \frac{\bxi_2+\bxi_1}{2},   \frac{\bzeta_2+ \bzeta_1}{\eps}  \big)
A \big( z, \frac{\bxi_2-\bxi_1}{2},   \frac{\bzeta_2- \bzeta_1}{\eps}  \big) \\
\nonumber
&&
\quad
 +K(z)  \phi^\eps ( \bzeta_1 ) \phi^\eps (\bxi_2)
A \big( z, \frac{\bzeta_2+\bxi_1}{2},  \frac{\bxi_2+ \bzeta_1}{\eps}   \big)
A \big( z, \frac{\bzeta_2-\bxi_1}{2},  \frac{\bxi_2- \bzeta_1}{\eps}  \big)
\\
&& 
\quad
 + R^\eps  (z ,\bxi_1,\bxi_2 ,  \bzeta_1 ,\bzeta_2 )   ,
\label{eq:propsci11}
\end{eqnarray}
where the functions $K$ and $A$  are defined by
\begin{eqnarray}
\label{def:K}
K(z) &:=& (2\pi)^8 \exp\Big(- \frac{k_0^2}{2} C({\bf 0}) z\Big) , \\
\nonumber
A(z,\bxi,\bzeta)  &:=&  \frac{1}{2(2\pi)^2}
 \int  \Big[  \exp \Big( \frac{k_0^2}{4} \int_0^z C\big( \bx + \frac{\bzeta}{k_0} z' \big) dz' \Big) -1\Big]\\
 && \times
   \exp \big( -i \bxi\cdot \bx  \big)
 d\bx  ,   
\label{def:A}
\end{eqnarray}
and the function $R^\eps $ satisfies
$$
\sup_{z \in [0,Z]} \| R^\eps (z,\cdot,\cdot,\cdot,\cdot) \|_{L^1(\RR^2\times \RR^2\times \RR^2\times \RR^2)} 
\stackrel{\eps \to 0}{\longrightarrow}  0  ,
$$
for any $Z>0$.
\end{proposition}
It follows from the proof given in Appendix \ref{app:b} that the function $\bxi \to A(z,\bxi,\bzeta)$ belongs to $L^1(\RR^2)$ 
and that its $L^1$-norm $\| A(z,\cdot,\bzeta)\|_{L^1(\RR^2)}$ is bounded uniformly in 
$\bzeta \in \RR^2$ and $z\in [0,Z]$. 
Therefore,  all terms in the right-hand side of (\ref{eq:propsci11}) are in 
$L^1(\RR^2\times \RR^2\times \RR^2\times \RR^2)$ with $L^1$-norms bounded uniformly in $\eps$ and $z \in [0,Z]$.
This proposition is important as many quantities of interest, such as the intensity correlation function, 
the scintillation index, or the variance of the Wigner transform of the wave field
that we will address in the next sections, 
can be expressed as integrals of $\tilde{\mu}^\eps$
against bounded functions. As a consequence we will be able to substitute $\tilde{\mu}^\eps$
with the right-hand side of (\ref{eq:propsci11}) without the remainder~$R^\eps$
in these integrals, and this substitution will allow us to give quantitative results.

\section{The Gaussian Summation Rule for the Centered Fourth-order Moments}
\label{sec:gsr} 
We consider the scaled fourth-order moment:
\begin{equation}
 \mu_4^\eps (z , \br; \bx_1,\bx_2,\by_1,\by_2 )
 := \mu_4 \Big( \frac{z}{\eps}, \frac{\br}{\eps}+\bx_1 , \frac{\br}{\eps}+\bx_2, \frac{\br}{\eps}+\by_1 , \frac{\br}{\eps}+\by_2\Big)  .
\end{equation}
The fourth-order moment can be expressed in terms of  $\tilde{\mu}^\eps$ as
\begin{eqnarray*}
 &&\mu_4^\eps (z , \br; \bx_1,\bx_2,\by_1,\by_2 )  =
\frac{1}{(2\pi)^8}
\iint \exp \Big(  i  \big(\bxi_1 \cdot \bq_1 + \bxi_2 \cdot \bq_2 +  \bzeta_1 \cdot \br_1^\eps +  \bzeta_2 \cdot \br_2   \big) \Big) \\
  && \quad \quad
  \times
   \exp \Big(  - i \frac{z}{k_0 \eps} (  \bxi_2\cdot \bzeta_2+ \bxi_1 \cdot \bzeta_1)
  \Big)
  \tilde{\mu}^\eps\big( z,\bxi_1,\bxi_2, \bzeta_1, \bzeta_2  \big)
  d\bzeta_1 d\bzeta_2 d\bxi_1 d\bxi_2 ,
\end{eqnarray*}
with
\begin{eqnarray*}
\br_1^\eps = 2 \frac{\br}{\eps} + \frac{\bx_1+\bx_2+\by_1+\by_2}{2}, \quad 
\bq_1 =  \frac{\bx_1+\bx_2-\by_1-\by_2}{2}, \\
\br_2 =  \frac{\bx_1-\bx_2+\by_1-\by_2}{2}, \quad 
\bq_2 =  \frac{\bx_1-\bx_2-\by_1+\by_2}{2} .
\end{eqnarray*}
Using Proposition \ref{prop:sci1}, the  fourth-order moment 
has the following form in the regime $\eps \to 0$:
\begin{eqnarray*}
&&  
\mu_4^\eps (z , \br; \bx_1,\bx_2,\by_1,\by_2 )\stackrel{\eps \to 0}{\longrightarrow}
  \frac{4 K(z)}{(2\pi)^8} \iint  d\balpha d\bbeta d\bzeta_2 d\bzeta_1
  e^{ -\frac{r_0^2}{2} (|\bzeta_2|^2+|\bzeta_1|^2) + 2 i \br \cdot \bzeta_1}  \\
  && \Big[ 
  \frac{r_0^8}{(2\pi)^4} e^{ -r_0^2(|\balpha|^2+|\bbeta|^2)} \\
&& +    \frac{r_0^6}{(2\pi)^3}  A(z,\balpha,\bzeta_2+\bzeta_1) e^{-r_0^2 |\bbeta|^2} e^{ i \balpha 
( \bx_1-\by_1 - \frac{\bzeta_2+\bzeta_1}{k_0}z)} \\
&& 
 +
  \frac{r_0^6}{(2\pi)^3} A(z,\bbeta,\bzeta_2-\bzeta_1) e^{-r_0^2 |\balpha|^2} e^{ i \bbeta \cdot 
  (\by_2-\bx_2-  \frac{\bzeta_2-\bzeta_1}{k_0}z)}\\
&&  +
  \frac{r_0^4}{(2\pi)^2} A(z,\balpha,\bzeta_2+\bzeta_1) A(z,\bbeta,\bzeta_2-\bzeta_1) 
e^{ i \balpha \cdot  (\bx_1-\by_1 - \frac{\bzeta_2+\bzeta_1}{k_0}z)+  
 i \bbeta \cdot (\by_2-\bx_2- \frac{\bzeta_2-\bzeta_1}{k_0}z )} 
  \Big] \\
&& \hspace*{1.1in}
+
 \frac{4 K(z)}{(2\pi)^8} \iint  d\balpha d\bbeta d\bxi_2 d\bzeta_1
 e^{ -\frac{r_0^2}{2} (|\bxi_2|^2+|\bzeta_1|^2) + 2 i \br \cdot \bzeta_1}  \\
 &&  \times \Big[ 
  \frac{r_0^6}{(2\pi)^3}  A(z,\balpha,\bxi_2+\bzeta_1) 
  e^{-r_0^2 |\bbeta|^2} e^{i \balpha \cdot(\bx_1-\by_2 - \frac{\bxi_2+\bzeta_1}{k_0}z)} \\
&& 
 +
  \frac{r_0^6}{(2\pi)^3}  A(z,\bbeta,\bxi_2-\bzeta_1) e^{-r_0^2 |\balpha|^2} e^{i \bbeta \cdot(\by_1-\bx_2 - \frac{\bxi_2-\bzeta_1}{k_0}z)} \\
&& +
  \frac{r_0^4}{(2\pi)^2} A(z,\balpha,\bxi_2+\bzeta_1) A(z,\bbeta,\bxi_2-\bzeta_1) 
 e^{i \balpha \cdot(\bx_1-\by_2 - \frac{\bxi_2+\bzeta_1}{k_0}z) + 
 i \bbeta \cdot(\by_1-\bx_2 - \frac{\bxi_2-\bzeta_1}{k_0}z) }
  \Big]   .
\end{eqnarray*}
Using the explicit form (\ref{def:A}) of $A$, 
this expression can be simplified  to
\begin{eqnarray}
\nonumber
&&\mu_4^\eps (z , \br; \bx_1,\bx_2,\by_1,\by_2 )\stackrel{\eps \to 0}{\longrightarrow} 
 -\exp \Big( - \frac{k_0^2 C({\bf 0}) z}{2} \Big) \exp\Big(- \frac{2 |\br|^2}{r_0^2} \Big)\\
\nonumber
&&\quad + 
 \Big[  \frac{r_0^2}{4 \pi} \int \exp \Big( \frac{k_0^2 }{4} \int_0^z C\big( \bzeta \frac{z'}{k_0}+\by_1-\bx_1\big) -C({\bf 0}) dz' 
 - \frac{r_0^2|\bzeta|^2}{4} + i \bzeta \cdot \br  \Big) d\bzeta\Big]\\
\nonumber
&&\quad \quad \times 
 \Big[  \frac{r_0^2}{4 \pi} \int \exp \Big( \frac{k_0^2 }{4} \int_0^z C\big( \bzeta \frac{z'}{k_0}+\by_2-\bx_2\big) -C({\bf 0}) dz' 
 - \frac{r_0^2|\bzeta|^2}{4} + i \bzeta \cdot \br  \Big) d\bzeta\Big]\\
\nonumber
&&\quad+ 
 \Big[  \frac{r_0^2}{4 \pi} \int \exp \Big( \frac{k_0^2 }{4} \int_0^z C\big( \bzeta \frac{z'}{k_0}+\by_2-\bx_1 \big) -C({\bf 0}) dz' 
 - \frac{r_0^2|\bzeta|^2}{4} + i \bzeta \cdot \br  \Big) d\bzeta\Big]\\
&& \quad \quad \times
 \Big[  \frac{r_0^2}{4 \pi} \int \exp \Big( \frac{k_0^2 }{4} \int_0^z C\big( \bzeta \frac{z'}{k_0}+\by_1-\bx_2 \big) -C({\bf 0}) dz' 
 - \frac{r_0^2|\bzeta|^2}{4} + i \bzeta \cdot \br  \Big) d\bzeta\Big]
  . \hspace*{0.3in}
 \label{eq:limitmueps4}
 \end{eqnarray}

For comparison, the scaled second-order moment defined by
\begin{equation}
\label{def:mu2eps}
\mu_2^{\eps}(z,\br;\bx,\by)  :=  
\mu\Big(\frac{z}{\eps},\frac{\br}{\eps} +\bx,\frac{\br}{\eps} +\by\Big) 
\end{equation}
is given by 
(see  (\ref{eq:transb12}) with $r_0\to r_0/\eps$,  $z \to z/\eps$, and $C \to \eps C$):
\begin{eqnarray}
\nonumber
\mu_2^{\eps}(z,\br;\bx,\by)  &=&
\frac{r_0^2 }{4 \pi \eps^2}  \int
\exp \Big(-\frac{\eps^2}{4 r_0^2} \big|  \bx-\by - \bzeta  \frac{z}{k_0\eps}\big|^2 -
\frac{r_0^2   |\bzeta|^2 }{4\eps^2} + i \bzeta \cdot   \big( \frac{\br}{\eps} +\frac{\bx+\by}{2} \big) \Big)\\
\nonumber&& \times \exp \Big( \frac{k_0^2 \eps}{4} \int_0^{z/\eps}
 {C}\big(  \bx-\by - \bzeta \frac{z'}{k_0}  \big)
- { C}( {\bf 0} ) dz' \Big)
 d \bzeta \\
 \nonumber&=&
\frac{r_0^2 }{4 \pi}  \int
\exp \Big(-\frac{\eps^2}{4 r_0^2} \big|   \bx-\by - \bzeta  \frac{z}{k_0}\big|^2 -
\frac{r_0^2   |\bzeta|^2 }{4} + i \bzeta \cdot  \big(  \br +\eps \frac{\bx+\by}{2} \big)  \Big)\\
&& \times \exp \Big( \frac{k_0^2}{4} \int_0^{z}
 {C}\big(  \bx-\by - \bzeta \frac{z'}{k_0}  \big)
- { C}( {\bf 0} ) dz' \Big)
 d \bzeta ,
\end{eqnarray}
so that in the limit $\eps \to 0$ : 
\begin{equation}
\mu_2^{\eps}(z,\br;\bx,\by)  \stackrel{\eps \to 0}{\longrightarrow}  
 \frac{r_0^2}{4 \pi} \int \exp \Big( \frac{k_0^2 }{4} \int_0^z C\big( \bzeta \frac{z'}{k_0}+\by-\bx \big) -C({\bf 0}) dz' 
 - \frac{r_0^2|\bzeta|^2}{4} + i \bzeta \cdot \br  \Big) d\bzeta.
 \label{eq:limitmueps2}
\end{equation}

The scaled first-order moment defined by
\begin{equation}
\mu_1^\eps (z,\br;\bx) : = \mu_1\Big(\frac{z}{\eps}, \frac{\br}{\eps}+\bx\Big)
\end{equation}
is given by (see (\ref{eq:expresmu1}) with $r_0\to r_0/\eps$,  $z \to z/\eps$, and $C \to \eps C$):
$$
\mu_1^\eps (z,\br;\bx) =   \frac{r_0^2}{{r_z^\eps}^2} \exp \Big( - \frac{k_0^2 C({\bf 0}) z}{8}\Big)
   \exp \Big( -\frac{| \br +\eps \bx|^2}{2 {r_z^\eps}^2 } \Big),\quad \quad  {r_z^\eps}^2 = r_0^2 \left( 1 +\frac{i\eps z}{k_0 r_0^2}  \right)  .
$$
In the limit $\eps \to 0$, we have
\begin{equation}
\mu_1^{\eps}(z,\br;\bx)  \stackrel{\eps \to 0}{\longrightarrow}  
\exp \Big( - \frac{k_0^2 C({\bf 0}) z}{8} \Big) \exp\Big(- \frac{ |\br|^2}{2r_0^2} \Big).
\label{eq:limitmueps1}
\end{equation}

As a consequence of (\ref{eq:limitmueps4}), (\ref{eq:limitmueps2}), and (\ref{eq:limitmueps1}),
we can check that, in the limit $\eps \to 0$, 
the Gaussian summation rule is satisfied.

\begin{proposition}
Under Hypothesis (\ref{hyp:h}),
in the scintillation regime $\eps \to 0$, we have
\begin{eqnarray}
\nonumber
\mu_4^\eps (z,\br;\bx_1,\bx_2,\by_1,\by_2) &=&\mu_2^\eps(z,\br;\bx_1,\by_1)  \mu_2^\eps(z,\br;\bx_2,\by_2) \\
\nonumber&&+
 \mu_2^\eps(z,\br;\bx_1,\by_2)  \mu_2^\eps(z,\br;\bx_2,\by_1) \\
&& -
    \mu_1^\eps(z,\br;\bx_1)   \mu_1^\eps(z,\br;\bx_2)  \overline{\mu_1^\eps(z,\br;\by_1)} \overline{  \mu_1^\eps(z,\br;\by_2)} ,
    \hspace*{0.2in}
\end{eqnarray}
in the sense that the  terms of this equation converge to  quantities that satisfy the Gaussian summation rule.
\end{proposition}
As noted in Section \ref{sec:main} this result is in agreement with the physical conjecture that 
a strongly scattered field  has Gaussian statistics.

\section{The Intensity Correlation Function}
\label{sec:intensity}
The intensity correlation function is usually defined by \cite[Eq.~(20.125)]{ishimaru}:
\begin{equation}
\label{def:gamma4}
\Gamma^{(4)}(z,\br,\bq)  =  \EE \Big[ \big| {u}\big(z,\br +\frac{\bq}{2}
\big)\big|^2 \big| {u}\big(z,\br -\frac{\bq}{2} \big)\big|^2  \Big] .
\end{equation}
Accordingly
we define the intensity correlation function  in our framework in the scintillation regime by
\begin{eqnarray}
\nonumber
\Gamma^{(4,\eps)}(z,\br,\bq) &:=&   \mu_4 \Big( \frac{z}{\eps},  \frac{\br}{\eps} +\frac{\bq}{2} ,   \frac{\br}{\eps} -\frac{\bq}{2} ,   \frac{\br}{\eps} +\frac{\bq}{2} ,   \frac{\br}{\eps} -\frac{\bq}{2}  \Big) \\
&=&   \mu \Big( \frac{z}{\eps},  \bq_1 = {\bf 0} , \bq_2= {\bf 0},  \br_1 = 2 \frac{\br}{\eps}   , \br_2= \bq   \Big) ,
\label{def:gamma4eps}
\end{eqnarray}
that is, the mid-point $\br/\eps$ is of the order of the initial beam width, and the off-set $\bq$ is of the order of the correlation length of the medium.
The  intensity correlation function can be expressed in terms of 
$\tilde{\mu}^\eps$ as
\begin{eqnarray*}
\Gamma^{(4,\eps)}(z,\br,\bq)  &=&  
\frac{1}{(2\pi)^8}
\iint \exp \Big(  
 2 i \frac{\bzeta_1 \cdot \br}{\eps} + i \bzeta_2 \cdot \bq - i \frac{z}{k_0 \eps} (  \bxi_2\cdot \bzeta_2+ \bxi_1 \cdot \bzeta_1)
  \Big) \\
  &&
  \times\tilde{\mu}^\eps\big( z,\bxi_1,\bxi_2, \bzeta_1, \bzeta_2  \big)
  d\bzeta_1 d\bzeta_2 d\bxi_1 d\bxi_2 .
\end{eqnarray*}
Using the result obtained in the previous section:
\begin{eqnarray}
\nonumber
&& \Gamma^{(4,\eps)}(z,\br,\bq)\stackrel{\eps \to 0}{\longrightarrow} 
 -\exp \Big( - \frac{k_0^2 C({\bf 0}) z}{2} \Big) \exp\Big(- \frac{2 |\br|^2}{r_0^2} \Big)\\
\nonumber
&&\quad + 
 \Big|  \frac{r_0^2}{4 \pi} \int \exp \Big( \frac{k_0^2 }{4} \int_0^z C\big( \bzeta \frac{z'}{k_0}\big) -C({\bf 0}) dz' 
 - \frac{r_0^2|\bzeta|^2}{4} + i \bzeta \cdot \br  \Big) d\bzeta\Big|^2\\
&&\quad+ 
 \Big|  \frac{r_0^2}{4 \pi} \int \exp \Big( \frac{k_0^2 }{4} \int_0^z C\big( \bzeta \frac{z'}{k_0}-\bq \big) -C({\bf 0}) dz' 
 - \frac{r_0^2|\bzeta|^2}{4} + i \bzeta \cdot \br  \Big) d\bzeta\Big|^2. \hspace*{0.2in}
 \label{eq:expresGamma4eps}
\end{eqnarray}
For comparison, the scaled mutual coherence function defined by
\begin{equation}
\label{def:gamma2eps}
\Gamma^{(2,\eps)}(z,\br,\bq)  := \mu_2\Big(\frac{z}{\eps}, \frac{\br}{\eps}+\frac{\bq}{2} , \frac{\br}{\eps}-\frac{\bq}{2} \Big) 
\end{equation}
satisfies in the limit $\eps \to 0$ : 
\begin{equation}
\label{eq:expresGamma2eps}
 \Gamma^{(2,\eps)}(z,\br,\bq)\stackrel{\eps \to 0}{\longrightarrow}  
 \frac{r_0^2}{4 \pi} \int \exp \Big( \frac{k_0^2 }{4} \int_0^z C\big( \bzeta \frac{z'}{k_0}-\bq \big) -C({\bf 0}) dz' 
 - \frac{r_0^2|\bzeta|^2}{4} + i \bzeta \cdot \br  \Big) d\bzeta.
\end{equation}

Before giving the result about the scintillation index,
we  briefly revisit the case of a plane wave, which corresponds to the limit case $r_0 \to \infty$
and which was already addressed in \cite{garniers3}.
We here find that, in the double limit $\eps \to 0$ and $r_0 \to \infty$:
\begin{eqnarray*}
\lim_{r_0 \to \infty} \lim_{\eps \to 0} \Gamma^{(2,\eps)}(z,\br,\bq)  
&=&    \exp  \Big(\frac{k_0^2 ( C(\bq)  - C({\bf 0}) ) z}{4}\Big)  ,\\
\end{eqnarray*} 
moreover, by (\ref{eq:expresGamma4eps})
\begin{eqnarray*}
\lim_{r_0 \to \infty} \lim_{\eps \to 0} \Gamma^{(4,\eps)}(z,\br,\bq)  
&=&  1- \exp \Big(-\frac{k_0^2 C({\bf 0}) z}{2}\Big)  
+  \exp  \Big(\frac{k_0^2 ( C(\bq)  - C({\bf 0}) ) z}{2}\Big)  ,
\end{eqnarray*}
which is the result obtained in \cite{garniers3}. Note that in \cite{garniers3} we first took the limit 
$r_0 \to \infty$,
and then $\eps \to 0$, while we here do the opposite. The two limits are exchangeable.
As discussed in \cite{garniers3}, this result shows in particular that the scintillation 
index, that is,  the variance of the intensity 
divided by the square of the mean intensity as defined below in
(\ref{eq:scint}),   is close to one when $k_0^2 C({\bf 0}) z\gg 1$.

We next consider the scintillation index in the general case of an initial Gaussian
beam as considered here. 
The expressions (\ref{eq:expresGamma4eps}) and (\ref{eq:expresGamma2eps}) allow us to describe 
 the scintillation index of the transmitted beam for the general case of an initial Gaussian beam with radius $r_0$. 

The scintillation index is usually defined as the square coefficient of variation of the intensity \cite[Eq. (20.151)]{ishimaru}:
$$
 S(z,\br)= 
\frac{ \EE \big[  \big| {u} (z,\br) \big|^4 \big]-
 \EE \big[   \big| {u} (z,\br)  \big|^2\big]^2 }
 { \EE \big[   \big| {u} (z,\br)   \big|^2\big]^2 } .
$$
In our framework, in the scintillation regime,
we define the scintillation index as:
\begin{equation}
\label{eq:scint}
 S^\eps(z,\br) := 
 \frac{\Gamma^{(4,\eps)}(z,\br,{\bf 0}) -\Gamma^{(2,\eps)}(z,\br,{\bf 0})^2 }{\Gamma^{(2,\eps)}(z,\br,{\bf 0})^2} .
 \end{equation}

\begin{proposition}
Under Hypothesis (\ref{hyp:h}),
the scintillation index (\ref{eq:scint})
has the following expression in the limit $\eps \to 0$:
 \begin{equation}
 S^\eps(z,\br) \stackrel{\eps \to 0}{\longrightarrow} 1 - \frac{ \exp\Big( -\frac{2 |\br|^2}{r_0^2} \Big)}{
 \Big|  \frac{1}{4 \pi} \int \exp \Big( \frac{k_0^2 }{4} \int_0^z C\big( \bu \frac{z'}{k_0 r_0}\big)  dz' 
 - \frac{|\bu|^2}{4} + i \bu \cdot \frac{\br}{r_0}  \Big) d\bu\Big|^2}
 .
\end{equation}
\end{proposition}

Let us consider the following form of the covariance function of the medium fluctuations:
$$
C(\bx) = C({\bf 0}) \tilde{C} \Big( \frac{|\bx|}{l_{\rm c}} \Big),
$$
with $\tilde{C}(0)=1$ and the width of the function $x \to \tilde{C}(x)$ is of order one.
For instance, we may consider  $\tilde{C}(x) = \exp(-x^2)$. 
Then the scintillation index at the beam center $\br={\bf 0}$ is
 \begin{equation}
 S^\eps(z,{\bf 0})\stackrel{\eps \to 0}{\longrightarrow} 1- \frac{4}{
 \Big|  \int_0^\infty \exp \Big( \frac{2z}{Z_{\rm sca}} \int_0^1 \tilde{C}\big( u \frac{z}{Z_{\rm c}} s\big)  ds
 - \frac{u^2}{4}  \Big) u du\Big|^2},
 \label{eq:scinnorm}
 \end{equation}
which is a function of $z/Z_{\rm sca}$ and $z/Z_{\rm c}$ only (or, equivalently, a 
function of $z/Z_{\rm sca}$ and $Z_{\rm c}/Z_{\rm sca}$ only), 
where $Z_{\rm sca}= \frac{8}{k_0^2 C({\bf 0})}$ and $Z_{\rm c} = k_0 r_0 l_{\rm c}$.
Here $Z_{\rm sca}$ is the scattering mean free path, since the mean field decays
exponentially at this rate:
$$
\EE \Big[ u \big( \frac{z}{\eps}, \frac{\bx}{\eps} \big) \Big] \stackrel{\eps \to 0}{\longrightarrow} 
\exp \Big( - \frac{|\bx|^2}{2 r_0^2} \Big) \exp \Big( - \frac{z}{Z_{\rm sca}} \Big),
$$  
as can be seen from (\ref{eq:expresmu1}). 
Moreover, 
$Z_{\rm c}$ is the typical propagation distance for which diffractive effects are of order one,
as shown in \cite[Eq.~ 4.4]{garniers1}. 
The function (\ref{eq:scinnorm}) is plotted in Figure \ref{fig:1}  in the case of  
Gaussian  correlations for  the medium fluctuations: 
$\tilde{C}(x) = \exp(-x^2)$.
It is interesting to note that, even if the propagation distance is larger than the scattering mean free
path, the scintillation index can be smaller than one if $Z_{\rm c}$ is small enough.

\begin{figure}
\begin{center}
\begin{tabular}{c}
\includegraphics[width=7.0cm]{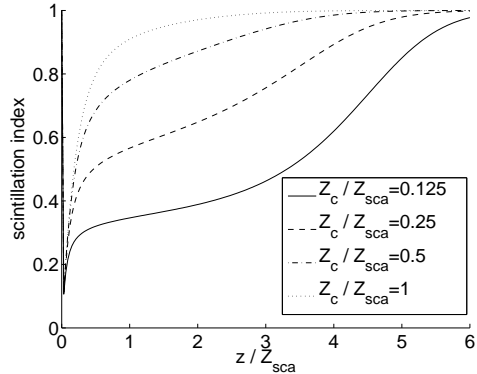}
\end{tabular}
\end{center}
\caption{Scintillation index at the beam center 
(\ref{eq:scinnorm}) as a function of the propagation distance for different values
of $Z_{\rm sca}$ and $Z_{\rm c}$. Here $\tilde{C}(x) = \exp(-x^2)$.
\label{fig:1} 
}
\end{figure}

In order to get more explicit expressions 
that facilitate  interpretation of  the results let us assume that $C(\bx)$ can be expanded as
\begin{equation}
\label{eq:expandC}
C(\bx)=C({\bf 0}) - \frac{\gamma}{2} |\bx|^2 + o (|\bx|^2), \quad \quad \bx \to {\bf 0}.
\end{equation}
When scattering is strong in the sense that the propagation distance
 is larger than the scattering mean free path $k_0^2 C({\bf 0}) z \gg 1$, 
 we have
 $$
K(z)^{1/2} A( z,\bxi,\bzeta) \simeq \frac{(2\pi)^4}{ \pi k_0^2 \gamma z}
\exp \Big( -\frac{\gamma z^3}{96} |\bzeta|^2 -\frac{2}{k_0^2 \gamma z} |\bxi|^2 + \frac{ iz}{2k_0}
\bzeta \cdot \bxi \Big) ,
 $$
 and  Eqs.~(\ref{eq:expresGamma4eps}) and (\ref{eq:expresGamma2eps}) can be simplified:
\begin{eqnarray}
\label{eq:corfieldplane}
\nonumber
&&\Gamma^{(2,\eps)}(z,\br,\bq) 
\stackrel{\eps \to 0}{\longrightarrow}
 \frac{r_0^2}{r_0^2 +\frac{\gamma z^3}{6} }\\
 && \quad \quad \times 
\exp\Big( - \frac{ |\br|^2}{r_0^2  +\frac{\gamma z^3}{6}} - \frac{k_0^2 \gamma z|\bq|^2}{8}
\frac{r_0^2 +\frac{\gamma z^3}{24}}{r_0^2 +\frac{\gamma z^3}{6}}
+ i  \frac{k_0\gamma z^2 \br \cdot \bq}{4(r_0^2 +\frac{\gamma z^3}{6})} \Big)
,\\
\nonumber
&&\Gamma^{(4,\eps)}(z,\br,\bq)
\stackrel{\eps \to 0}{\longrightarrow}
 \frac{r_0^4}{\big(r_0^2 +\frac{\gamma z^3}{6}\big)^2}\\
 && \quad \quad \times
\exp\Big( - \frac{2 |\br|^2}{r_0^2  +\frac{\gamma z^3}{6}} \Big)
\Big[ 1 + 
\exp\Big(  - \frac{k_0^2 \gamma z|\bq|^2}{4}
\frac{r_0^2 +\frac{\gamma z^3}{24}}{r_0^2 +\frac{\gamma z^3}{6}} \Big)
\Big].
\label{eq:corintplane}
\end{eqnarray}
This shows that, in the regime $\eps \to 0$ and $k_0^2 C({\bf 0})z \gg 1$:\\
- The beam radius is $R_z$ with
\begin{equation}
\label{eq:beamradius}
R_z^2 :=  r_0^2 +\frac{\gamma z^3}{6} .
\end{equation}
- The correlation radius of the intensity distribution is $\rho_z$ with
\begin{equation}
\label{eq:corrradius}
\rho_z^2 := \frac{4}{k_0^2 \gamma z }\frac{r_0^2 +\frac{\gamma z^3}{6}}{r_0^2 +\frac{\gamma z^3}{24}} ,
\end{equation}
which is of the same order as the correlation radius of the field
(compare the $\bq$-dependence of (\ref{eq:corfieldplane}) and (\ref{eq:corintplane})).\\
- The scintillation index is close to one:
\begin{equation}
\label{eq:scin:strong}
 S^\eps(z,\br) = \frac{\Gamma^{(4,\eps)}(z,\br,{\bf 0}) - \Gamma^{(2,\eps)}(z,\br,{\bf 0})^2}{\Gamma^{(2,\eps)}(z,\br,{\bf 0})^2}
 \simeq 1 .
\end{equation}
This observation is consistent with the physical intuition that, in the strongly scattering regime
$z / Z_{\rm sca} \gg 1$, the wave field is conjectured to have zero-mean complex circularly symmetric Gaussian statistics,
and therefore the intensity is expected to have exponential (or Rayleigh) distribution \cite{fante75,ishimaru},
in agreement with (\ref{eq:scin:strong}).

\section{Stability of the Wigner Transform of the Field}
\label{sec:wigner}
The Wigner transform of the transmitted field is defined by
\begin{equation}
{W}^\eps ( z,\br,\bxi) := 
\int 
\exp \big( - i  \bxi \cdot \bq  \big)
  {u}\big(\frac{z}{\eps},  \frac{\br}{\eps} + \frac{\bq}{2} \big)
 \overline{u} \big( \frac{z}{\eps},   \frac{\br}{\eps} - \frac{\bq}{2}\big)   d \bq  .
\end{equation}  
It is an important quantity that can be interpreted as the angularly-resolved  wave energy density
(note, however, that it is real-valued but not always non-negative valued).
Remember that the initial source is (\ref{def:feps}).
This means that the Wigner transform is observed at a mid point $\br/\eps$ that is at the scale of the initial beam radius,
while the offset $\bq$ is observed at the scale of the correlation length of the medium.
In the homogeneous case, we find  
\begin{equation}
\label{eq:wignerhomo1}
{W}^\eps ( z,\br,\bxi) \mid_{\rm homo} = 
\frac{4 \pi r_0^2}{\eps^2} \exp \big( - \frac{|\bxi|^2 r_0^2}{\eps^2}  
 - \frac{|\br - \bxi z/k_0|^2}{r_0^2} \big)  ,
\end{equation}
which is concentrated in a narrow cone in $\bxi$. 
Indeed the $\bxi$-dependence of the Wigner transform reflects the angular diversity of the beam.
In the limit $\eps \to 0$, we have 
\begin{equation}
{W}^\eps ( z,\br,\bxi) \mid_{\rm homo} \stackrel{\eps \to 0}{\longrightarrow}
(2\pi)^2 \delta(\bxi) 
\exp \big( - \frac{|\br|^2}{r_0^2} \big)  ,
\end{equation}
in the sense that, for any continuous and bounded function $\psi$,
$$
\iint {W}^\eps ( z,\br,\bxi) \mid_{\rm homo} \psi(\br,\bxi) d\br d\bxi \stackrel{\eps \to 0}{\longrightarrow} 
(2\pi)^2\int \exp \Big( - \frac{|\br|^2}{r_0^2} \Big)\psi(\br,{\bf 0} ) d\br  .
$$

In the random case,
the $\bxi$-dependence of the Wigner transform depends on the angular diversity of the initial beam
but also on the scattering by the random medium, which dramatically broadens  it because the correlation length
of the medium is smaller than the initial beam width. As a result
(see (\ref{solWT1b})  with $r_0\to r_0/\eps$,  $z \to z/\eps$, and $C \to \eps C$), 
the expectation of the Wigner transform is:
\begin{eqnarray}
\nonumber
\EE[ {W}^\eps ( z,\br,\bxi)] & = & 
\frac{r_0^2}{4\pi \eps^2} \iint \exp \Big(  - \frac{r_0^2|\bzeta|^2}{4 \eps^2} 
-\frac{\eps^2 |\bq|^2}{4r_0^2} + i \frac{\bzeta}{\eps} \cdot \big( \br  - \frac{\bxi z}{k_0}  \big) - i \bxi \cdot  \bq  \Big) \\
\nonumber
&& \times \exp \Big( \frac{k_0^2 \eps}{4} \int_0^{z/\eps} C\big( \bq+\bzeta \frac{z'}{k_0} \big) -C({\bf 0}) dz' \Big) d\bzeta  d \bq \\
\nonumber
&=&\frac{r_0^2}{4\pi} \iint \exp \Big(  - \frac{r_0^2|\bzeta|^2}{4} 
-\frac{\eps^2 |\bq|^2}{4r_0^2} + i \bzeta  \cdot \big( \br  - \frac{\bxi z}{k_0}  \big) - i \bxi \cdot  \bq  \Big) \\
&& \times \exp \Big( \frac{k_0^2}{4} \int_0^{z} C\big( \bq+\bzeta \frac{z'}{k_0} \big) -C({\bf 0}) dz' \Big) d\bzeta  d \bq ,
\end{eqnarray}
so that in the limit $\eps \to 0$ it is given by
\begin{eqnarray}
\nonumber
\EE[ {W}^\eps ( z,\br,\bxi)] & \stackrel{\eps \to 0}{\longrightarrow}& 
\frac{r_0^2}{4\pi} \iint \exp \Big(  - \frac{r_0^2|\bzeta|^2}{4} + i \bzeta \cdot \br 
 -i \bxi \cdot \big(\bq + \bzeta  \frac{z}{k_0} \big) \Big) \\
&& \times \exp \Big( \frac{k_0^2 }{4} \int_0^z C\big( \bq+\bzeta \frac{z'}{k_0} \big) -C({\bf 0}) dz' \Big) d\bzeta  d \bq .
\end{eqnarray}
More precisely, the mean Wigner transform can be split into two pieces: a narrow cone and a broad cone in $\bxi$:
\begin{eqnarray}
\nonumber
&&\EE[ {W}^\eps ( z,\br,\bxi)] \stackrel{\eps \to 0}{\longrightarrow} 
\frac{K(z)^{1/2}}{(2\pi)^2 } \delta(\bxi)  \exp \Big( - \frac{|\br|^2}{r_0^2} \Big)  \\
&& \quad \quad + \frac{r_0^2  K(z)^{1/2}}{(2\pi)^3} \int \exp\Big( 
 -  \frac{r_0^2 |\bzeta|^2}{4} + i \bzeta \cdot \big( \br - \bxi   \frac{z}{k_0} \big) \Big)
 A(z,\bxi,\bzeta)  d\bzeta 
 .
 \label{eq:meanwigner2}
\end{eqnarray}
The narrow cone is the contribution of the coherent transmitted wave components and it decays exponentially
with the propagation distance (see the expression (\ref{def:K}) for  $K(z)$). 
The broad cone is the contributions of the incoherent scattered waves and it becomes dominant when the propagation
distance becomes so large that $k_0^2C({\bf 0}) z\gg1$.

It is known that the Wigner transform is not statistically stable,
and that it is necessary to smooth it (that is to say, to convolve it with a kernel)
to get a quantity that can be measured in a statistically stable way
(that is to say, the Wigner transform for one typical realization is approximately equal to its expected value) \cite{balMMS,PRS07}.
Our goal in this section is to quantify this statistical stability.

Let us consider two positive parameters $r_{\rm s}$ and $\xi_{\rm s}$ and define the smoothed Wigner transform:
\begin{equation}
{W}^\eps_{\rm s} ( z,\br,\bxi) = 
\frac{1}{(2\pi)^2 \eps^2 r_{\rm s}^2 \xi_{\rm s}^2} \iint
{W}^\eps ( z,\br-\br',\bxi-\bxi') \exp \Big(- \frac{|\br'|^2}{2\eps^2 r_{\rm s}^2} - \frac{|\bxi'|^2}{2 \xi_{\rm s}^2} \Big)
d\br' d\bxi'   .
\label{def:widetildeWseps}
\end{equation}
In view of the form of the Wigner transform in the homogeneous case in (\ref{eq:wignerhomo1}) this is indeed the natural
scales at which to smooth.
The expectation of the smoothed Wigner transform is in the limit $\eps \to 0$:
\begin{eqnarray}
\nonumber
\EE\big[ {W}^\eps_{\rm s} ( z,\br,\bxi)\big] & \stackrel{\eps \to 0}{\longrightarrow} & 
\frac{r_0^2}{4\pi} \iint \exp\Big( 
 - 
\frac{r_0^2 |\bzeta|^2 }{4}
- \frac{\xi_{\rm s}^2 |\bq + \bzeta  \frac{z}{k_0}|^2}{2} 
 -i \bxi \cdot \big(\bq + \bzeta  \frac{z}{k_0} \big) \Big)\\
&& \times \exp \Big(  i \bzeta \cdot \br + \frac{k_0^2 }{4} \int_0^z C\big( \bq
+\bzeta \frac{z'}{k_0} \big) -C({\bf 0}) dz' \Big)  d\bzeta d\bq .
\end{eqnarray}
It can also be written as follows.
\begin{proposition}
\label{prop:smooth1}
Under Hypothesis (\ref{hyp:h}),
the expectation of the smoothed Wigner transform (\ref{def:widetildeWseps}) satisfies, in the scintillation regime $\eps \to 0$:
\begin{eqnarray}
\nonumber
\EE\big[ {W}^\eps_{\rm s} ( z,\br,\bxi)\big] & \stackrel{\eps \to 0}{\longrightarrow} & 
\frac{K(z)^{1/2}}{(2\pi)^3 \xi_{\rm s}^2} 
 \exp \Big( -\frac{|\bxi|^2}{2\xi_{\rm s}^2} \Big) \exp \Big( -\frac{|\br|^2}{r_0^2} \Big)  \\
 && \hspace*{-1.2in}
+ \frac{K(z)^{1/2} r_0^2}{(2\pi)^4\xi_{\rm s}^2} 
  \iint A(z,\bxi',\bzeta)
\exp \Big( - \frac{r_0^2 |\bzeta|^2}{4} - \frac{|\bxi'-\bxi|^2}{2 \xi_{\rm s}^2} +i \bzeta \cdot \big( \br- \bxi'\frac{z}{k_0}\big) 
 \Big)  d\bzeta d\bxi' ,  \hspace*{0.3in}
 \label{eq:meanwignersmooth}
\end{eqnarray}
where $K$ and $A$ are defined by (\ref{def:K}) and (\ref{def:A}).
\end{proposition}
The first term in (\ref{eq:meanwignersmooth}) is a narrow cone in $\bxi$ around $\bxi={\bf 0}$ corresponding to coherent wave components
and the second term is a broad cone in $\bxi$ corresponding to incoherent wave components.
Note that the expectation of the smoothed Wigner transform
is independent on $r_{\rm s}$ as the smoothing in $\br$ vanishes in the limit $\eps \to 0$. However the 
smoothing in $\br$ plays an important role in the control of the fluctuations of the Wigner transform.
We will analyze the variance of the smoothed Wigner transform
and  its dependence on the smoothing parameters $r_{\rm s}$ and~$\xi_{\rm s}$.

The second moment of the smoothed Wigner transform is
\begin{eqnarray*}
&&
\EE[ {W}^\eps_{\rm s} ( z,\br,\bxi)^2] =
\frac{1}{(2\pi)^2 \eps^4 r_{\rm s}^4} \iint \exp\Big( 
 -  \frac{|\br_{\rm s}|^2+|\br_{\rm s}'|^2}{2\eps^2 r_{\rm s}^2}  -\frac{\xi_{\rm s}^2 (|\bq|^2+|\bq'|^2)}{2} \Big)
\\
&&  \quad\times \mu \Big( \frac{z}{\eps} ,  
\frac{\bq+\bq'}{2}, 
\frac{\bq-\bq'}{2},
\frac{2\br+\br_{\rm s}+\br_{\rm s}'}{\eps},
\frac{\br_{\rm s}-\br_{\rm s}'}{\eps}\Big) \\
&& \quad \times  \exp\big(  - i \bxi \cdot (\bq+\bq') \big)
  d\bq d\bq' d\br_{\rm s} d\br_{\rm s}'  ,
\end{eqnarray*}
which gives, using   (\ref{eq:renormhatM2}),  (\ref{eq:fourier}),
 and   (\ref{def:MtildeMeps}):  
\begin{eqnarray*}
&&\EE \big[ {W}^\eps_{\rm s} ( z,\br,\bxi)^2\big] =
  \frac{1}{(2\pi)^6 \xi_{\rm s}^4 }   \iint \exp\Big( 
  - r_{\rm s}^2 |\bzeta_1|^2  -r_{\rm s}^2 |\bzeta_2|^2
 -\frac{|\bxi_1-2 \bxi|^2}{4 \xi_{\rm s}^2 }   -\frac{|\bxi_2 |^2}{4 \xi_{\rm s}^2} \Big)
\\
&& \quad \times \exp\Big(   2 i \frac{\bzeta_1 \cdot \br}{\eps} -i \frac{z}{k_0 \eps} \big(\bzeta_1\cdot \bxi_1+
 \bzeta_2\cdot \bxi_2  \big) \Big)\tilde{\mu}^\eps( z,\bxi_1,\bxi_2, \bzeta_1, \bzeta_2)
  d\bzeta_1 d\bzeta_2 d\bxi_1 d\bxi_2  .
\end{eqnarray*} 
Using Proposition \ref{prop:sci1} and the fact   that $A(z,-\bxi,-\bzeta)=A(z,\bxi,\bzeta)$,
we get the following proposition.
\begin{proposition}
\label{prop:smooth2}
Under Hypothesis (\ref{hyp:h}),
the second moment of the smoothed Wigner transform (\ref{def:widetildeWseps}) is, in the scintillation regime $\eps \to 0$:
\begin{eqnarray}
\nonumber
&& 
\EE \big[ {W}^\eps_{\rm s} ( z,\br,\bxi)^2 \big] \stackrel{\eps \to 0}{\longrightarrow}
\frac{K(z)}{(2\pi)^6  \xi_{\rm s}^4 } \exp \Big( -\frac{|\bxi|^2}{\xi_{\rm s}^2} \Big)
\exp \Big( -\frac{2 |\br|^2}{r_0^2} \Big)\\
\nonumber
&& +
 \frac{r_0^4 K(z)}{(2\pi)^8 \xi_{\rm s}^4 } \iint d\bxi_1 d\bzeta_1 
  e^{i \bzeta_1\cdot (2 \bx-\frac{z}{k_0} \bxi_1)
 -\frac{r_0^2 |\bzeta_1|^2}{2} - \frac{|\bxi_1-2\bxi|^2}{4 \xi_{\rm s}^2} }\\
\nonumber
 && \times \bigg\{
 4 e^{- \frac{|\bxi_1|^2}{4\xi_{\rm s}^2}}  \int e^{- i \frac{z}{k_0} \bxi_1 \cdot \bzeta_2  - \frac{r_0^2 |\bzeta_2|^2}{2}}
 A(z,\bxi_1,\bzeta_2+\bzeta_1) d\bzeta_2\\
\nonumber
&&  +   \iint e^{- \frac{|\bxi_2|^2}{4 \xi_{\rm s}^2} - i \frac{z}{k_0} \bxi_2 \cdot \bzeta_2  - \frac{r_0^2 |\bzeta_2|^2}{2}}
 A\big(z,\frac{\bxi_2+\bxi_1}{2},\bzeta_2+\bzeta_1\big) \\
\nonumber &&\hspace*{1.1in} \times 
A\big(z,\frac{\bxi_2-\bxi_1}{2},\bzeta_2-\bzeta_1\big) 
 d\bxi_2 d\bzeta_2\\
\nonumber
&&  + 
4 e^{-r_{\rm s}^2 |\bxi_1|^2 }  \int e^{- i \frac{z}{k_0} \bxi_1 \cdot \bxi_2  - \frac{r_0^2 |\bxi_2|^2}{2}}
 A(z,\bxi_1,\bxi_2+\bzeta_1) d\bxi_2\\
 \nonumber
 && +   \iint e^{-  r_{\rm s}^2 |\bzeta_2|^2  - i \frac{z}{k_0} \bxi_2 \cdot \bzeta_2  - \frac{r_0^2 |\bxi_2|^2}{2}}
 A\big(z,\frac{\bzeta_2+\bxi_1}{2},\bxi_2+\bzeta_1\big) \\
 &&\hspace*{1.1in} \times 
A\big(z,\frac{\bzeta_2-\bxi_1}{2},\bxi_2-\bzeta_1\big) 
 d\bxi_2 d\bzeta_2 \bigg\} .
\end{eqnarray}
\end{proposition}
This is an exact expression but, as it involves four-dimensional integrals, it is complicated to interpret it.
This expression becomes simple in the strongly scattering regime $k_0^2 C( {\bf 0} )z\gg 1$, because then 
$A(z,\bxi,\bzeta)$ takes a Gaussian form and all integrals can be evaluated.
To get  more explicit expressions in the discussion of the results  we here again assume 
 that $C(\bx)$ can be expanded as (\ref{eq:expandC}).
When $k_0^2 C({\bf 0}) z \gg 1$, we have
\begin{eqnarray}
\nonumber
\EE\big[ {W}^\eps_{\rm s} ( z,\br,\bxi)\big]  &\stackrel{\eps \to 0}{\longrightarrow}& 
\frac{8 \pi }{k_0^2 \gamma z } \frac{r_0^2}{ (r_0^2 +\frac{\gamma z^3}{24} )(1+\frac{4 \xi_{\rm s}^2}{k_0^2 \gamma z})
+\frac{z^2 \xi_{\rm s}^2}{2k_0^2}}
\\
&& 
\times \exp \bigg(  - \frac{\Big| \br - \frac{z \bxi}{2k_0 ( 1 +\frac{4 \xi_{\rm s}^2}{k_0^2 \gamma z})} \Big|^2}
{r_0^2  +\frac{\gamma z^3}{24} + \frac{ \frac{z^2\xi_{\rm s}^2}{2k_0^2}}{1+\frac{4 \xi_{\rm s}^2}{k_0^2 \gamma z} }}
- \frac{2 |\bxi|^2}{k_0^2 \gamma z + 4 \xi_{\rm s}^2} \bigg)
\end{eqnarray}
and
\begin{eqnarray}
\nonumber
\EE \big[ {W}^\eps_{\rm s} ( z,\br,\bxi)^2\big]  &\stackrel{\eps \to 0}{\longrightarrow}& 
\lim_{\eps \to 0}\EE\big[ {W}^\eps_{\rm s} ( z,\br,\bxi)\big] ^2
\left( 1 +
  \frac{
(r_0^2 +\frac{\gamma z^3}{24} )(1
+\frac{4\xi_{\rm s}^2 }{k_0^2 \gamma z})
+\frac{z^2 \xi_{\rm s}^2}{2k_0^2}
}{ (r_0^2 +\frac{\gamma z^3}{24} )(4r_{\rm s}^2\xi_{\rm s}^2 
+\frac{4\xi_{\rm s}^2 }{k_0^2 \gamma z})
+\frac{z^2 \xi_{\rm s}^2}{2k_0^2}} 
\right) .
\end{eqnarray}

The coefficient of variation ${C}^\eps_{\rm s}$ of the smoothed Wigner transform is defined by:
\begin{equation}
{C}_{\rm s}^\eps (z, \br,\bxi) := \frac{\sqrt{ \EE[ {W}^\eps_{\rm s} ( z,\br,\bxi)^2]  -\EE[ {W}^\eps_{\rm s} ( z,\br,\bxi)]^2}}{\EE[ {W}^\eps_{\rm s} ( z,\br,\bxi)]} .
\end{equation}
We then get the following expression for the coefficient of variation in the strongly scattering regime
$k_0^2 C({\bf 0} )z \gg 1$.

\begin{corollary}
Under the same hypotheses as in Propositions \ref{prop:smooth1}
 and \ref{prop:smooth2},
if additionally $k_0^2 C({\bf 0} )z \gg 1$ and $C$ can be expanded as (\ref{eq:expandC}), 
then the coefficient of variation  of the smoothed Wigner transform  (\ref{def:widetildeWseps}) satisfies
\begin{equation}
{C}^\eps_{\rm s}(z,\br,\bxi) \stackrel{\eps \to 0}{\longrightarrow}
\left( \frac{
(r_0^2 +\frac{\gamma z^3}{24} )(1
+\frac{4\xi_{\rm s}^2 }{k_0^2 \gamma z})
+\frac{z^2 \xi_{\rm s}^2}{2k_0^2}
}{ (r_0^2 +\frac{\gamma z^3}{24} )(4r_{\rm s}^2\xi_{\rm s}^2 
+\frac{4\xi_{\rm s}^2 }{k_0^2 \gamma z})
+\frac{z^2 \xi_{\rm s}^2}{2k_0^2}}  \right)^{1/2}
=
\left( \frac{  \frac{1}{\xi_{\rm s}^2 \rho_z^2} +1 }
 { \frac{ 4 r_{\rm s}^2}{\rho_z^2} + 1 }  \right)^{1/2}
,
\label{eq:Cepsstrong}
\end{equation}  
where $\rho_z$ is the correlation radius (\ref{eq:corrradius}). 
\end{corollary}
Note that the coefficient of variation
becomes independent of $\br$ and $\bxi$.
Eq.~(\ref{eq:Cepsstrong}) 
 is a simple enough formula to help determining the smoothing parameters $\xi_{\rm s}$ and $r_{\rm s}$
that are needed to reach a given value for the coefficient of variation.
The coefficient of variation is plotted in Figure \ref{fig:2}, which exhibits the line $2 \xi_{\rm s} r_{\rm s} =1$
separating the two regions where the coefficient of variation is larger or smaller than one.

\begin{figure}
\begin{center}
\begin{tabular}{c}
\includegraphics[width=7.0cm]{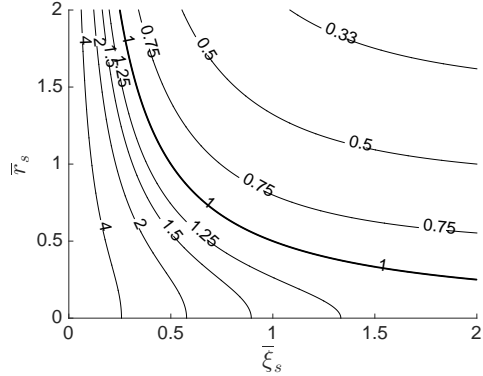}
\end{tabular}
\end{center}
\caption{Contour levels of the coefficient of variation (\ref{eq:Cepsstrong}) of the smoothed Wigner transform.
Here $ \overline{r}_{\rm s} = r_{\rm s} /\rho_z$
and $ \overline{\xi}_{\rm s} = \xi_{\rm s} \rho_z$.
The contour level $1$ is $2\overline{\xi}_{\rm s} \overline{r}_{\rm s} =1$.
\label{fig:2} 
}
\end{figure}

For $2 \xi_{\rm s} r_{\rm s} =1$, we have $\lim_{\eps \to 0} {C}^\eps_{\rm s}(z,\br,\bxi)=1$. 
For $2 \xi_{\rm s} r_{\rm s} <1$ (resp. $>1$) we have $\lim_{\eps \to 0} {C}^\eps_{\rm s}(z,\br,\bxi)>1$ (resp. $<1$).
The curve   $2 \xi_{\rm s} r_{\rm s} =1$ determines the region where the coefficient of variation
of ${W}^\eps_{\rm s} ( z,\br,\bxi)$ is smaller or larger than one (in the limit $\eps \to 0$).
The critical value $r_{\rm s}=1/(2\xi_{\rm s})$ is indeed special.
In this  case, the smoothed Wigner transform (\ref{def:widetildeWseps}) can be written as the double convolution
of the Wigner transform $W^\eps$ of the random field $u(\frac{z}{\eps},\cdot)$ 
with the Wigner transform 
$$
W^\eps_{\rm g} (\br,\bxi) := \int \exp \big( -i \bxi \cdot \bq \big) u_{\rm g} \big(\frac{\br}{\eps} +\frac{\bq}{2}\big)
\overline{u_{\rm g}} \big(\frac{\br}{\eps} -\frac{\bq}{2}\big) d\bq
$$
of the Gaussian state 
$$
u_{\rm g}(\bx)  := \exp\big( - \xi_{\rm s}^2  |\bx|^2 \big) ,
$$
since we have
$$
{W}^\eps_{\rm g} (\br,\bxi) =  \frac{2 \pi }{\xi_{\rm s}^2 }  \exp \Big( - 2 \frac{\xi_{\rm s}^2 |\br|^2}{\eps^2}  - \frac{ |\bxi|^2}{ 2\xi_{\rm s}^2} \Big) ,
$$
and therefore
$$
{W}^\eps_{\rm s} (z,\br,\bxi)  
 = \frac{4\xi_{\rm s}^2}{(2\pi)^3 \eps^2} \iint 
W^\eps(z,\br-\br',\bxi-\bxi') W^\eps_{\rm g} (\br',\bxi') d\br' d\bxi' ,
$$
for $r_{\rm s}=1/(2\xi_{\rm s})$.
It is known that the convolution of a Wigner transform with a kernel that is itself the Wigner transform 
of a function  (such as a Gaussian) is nonnegative real valued 
(the smoothed Wigner transform obtained with the Gaussian $W^\eps_{\rm g}$
is sometimes called Husimi function)
\cite{cartwright76,manfredi00}.
This can be shown easily in our case as the smoothed Wigner transform can be written as  
\begin{equation}
\label{eq:Wepsrsqs}
{W}^\eps_{\rm s} (z,\br,\bxi)     = \frac{2\xi_{\rm s}^2}{\pi  } \Big| \int
\exp \big(   i  {\bxi\cdot \br'} \big) \overline{u_{\rm g}} \big( \br' \big)
u \big( \frac{z}{\eps}, \frac{\br}{\eps}  - \br'\big) d\br' \Big|^2 ,
\end{equation} 
for $r_{\rm s}=1/(2\xi_{\rm s})$.
From this representation formula of ${W}^\eps_{\rm s} $ valid for $r_{\rm s}=1/(2 \xi_{\rm s})$,
we can see that it is the square modulus of a linear functional of $u(\frac{z}{\eps},\cdot)$.
The physical intuition that  $u(\frac{z}{\eps},\cdot)$ has circularly symmetric complex Gaussian statistics 
in strongly scattering media then predicts
that ${W}^\eps_{\rm s} (z,\br,\bxi)$ should have an exponential (or Rayleigh) distribution, because
the sum of the squares of two independent real-valued Gaussian random variables has an exponential
distribution. This is indeed consistent with our theoretical finding that $\lim_{\eps \to 0} {C}^\eps_{\rm s}=1$
for $r_{\rm s}=1/(2 \xi_{\rm s})$.  
In fact the situation with complex scattering giving a field that has centered circularly symmetric
Gaussian statistics is exactly what motivates the name ``scintillation regime''  with
unit relative intensity fluctuations.

If $  r_{\rm s}> 1/(2\xi_{\rm s})$, by observing  that 
$$
\exp \Big( - \frac{|\br|^2}{ 2 \eps^2 r_{\rm s}^2} \Big) =
\Psi^\eps (\br) *_\br
\exp \Big( - \frac{2 \xi_{\rm s}^2 |\br|^2}{\eps^2} \Big) ,
$$
where $*_\br$ stands for the convolution product in $\br$:
$$
\Psi^\eps (\br) *_\br f(\br) = \int \Psi^\eps (\br-\br')   f(\br') d\br'  ,
$$
and the function $\Psi^\eps$ is defined by
$$
\Psi^\eps  (\br) :=  \frac{ 8 \xi_{\rm s}^4 r_{\rm s}^2}{\pi \eps^2 ( 4 \xi_{\rm s}^2 r_{\rm s}^2 -1)}
\exp \Big( - \frac{2 \xi_{\rm s}^2 |\br|^2}{(4 \xi_{\rm s}^2 r_{\rm s}^2-1) \eps^2} \Big)  ,
$$
we observe that the smoothed Wigner transform (\ref{def:widetildeWseps}) can be expressed as:
\begin{equation}
{W}^\eps_{\rm s} (z,\br,\bxi)   = 
\Psi^\eps(\br) *_\br \left(
\frac{2\xi_{\rm s}^2}{\pi  } \Big| \int
\exp \big(   i  {\bxi \cdot \br'} \big) \overline{u_{\rm g}} \big( \br' \big)
u \big( \frac{z}{\eps}, \frac{\br}{\eps}  - \br'\big) d\br' \Big|^2 \right) ,
\end{equation} 
for $r_{\rm s}>1/(2\xi_{\rm s})$.
From this representation formula for  ${W}^\eps_{\rm s} $ valid for $r_{\rm s}>1/(2 \xi_{\rm s})$,
we can see that it is nonnegative valued and that it is a local average of (\ref{eq:Wepsrsqs}),
which has a unit coefficient of variation in the strongly scattering scintillation regime.
That is why the coefficient of variation of the smoothed Wigner transform
is smaller than one when $  r_{\rm s}> 1/(2\xi_{\rm s})$.

Finally, 
it is possible to take $r_{\rm s}=0$ in (\ref{def:widetildeWseps}), which corresponds to the absence of smoothing in $\br$:
$$
{W}^\eps_{\rm s}(z,\br,\bxi)  = \frac{1}{2\pi \xi_{\rm s}^2} \int W^\eps(z,\br,\bxi-\bxi') 
\exp \Big( -\frac{|\bxi'|^2}{2 \xi_{\rm s}^2} \Big) d\bxi' ,
$$
for $r_{\rm s}=0$.
We then get 
\begin{eqnarray*}
\nonumber
{\rm Var}\big( {W}^\eps_{\rm s} ( z,\br,\bxi) \big) &
 \stackrel{\eps \to 0}{\longrightarrow} 
&
 \frac{ \left( \frac{8 \pi  r_0^2 }{k_0^2 \gamma z } \right)^2}
{ \left( (r_0^2 +\frac{\gamma z^3}{24} )(1+\frac{4 \xi_{\rm s}^2}{k_0^2 \gamma z})
+\frac{z^2 \xi_{\rm s}^2}{2k_0^2}\right)   \left( (r_0^2 +\frac{\gamma z^3}{24} )(\frac{4 \xi_{\rm s}^2}{k_0^2 \gamma z})
+\frac{z^2 \xi_{\rm s}^2}{2k_0^2}\right)   }  \\
&&
\times  \exp \bigg(  - \frac{2 \Big| \br - \frac{z \bxi}{2k_0 ( 1 +\frac{4 \xi_{\rm s}^2}{k_0^2 \gamma z})} \Big|^2}
{r_0^2  +\frac{\gamma z^3}{24} + \frac{ \frac{z^2\xi_{\rm s}^2}{2k_0^2}}{1+\frac{4 \xi_{\rm s}^2}{k_0^2 \gamma z} }}
- \frac{4 |\bxi|^2}{k_0^2 \gamma z + 4 \xi_{\rm s}^2} \bigg)
\end{eqnarray*}
and
$$
{C}^\eps_{\rm s}(z,\br,\bxi)   \stackrel{\eps \to 0}{\longrightarrow}
  \sqrt{ 1 +     (\xi_{\rm s} \rho_z)^{-2} }  ,
$$ 
for $r_{\rm s}=0$.
If, additionally, we let $\xi_{\rm s} \to \infty$, then we find
\begin{eqnarray*}
&&\lim_{ \xi_{\rm s} \to \infty}  \lim_{\eps \to 0}
 \frac{\xi_{\rm s}^2}{2\pi}  \EE \big[ {W}^\eps_{\rm s} ( z,\br,\bxi) \big]  =
\frac{r_0^2}{ r_0^2 +\frac{\gamma z^3}{6} }  
\exp \Big(  - \frac{ | \br |^2}
{r_0^2  +\frac{\gamma z^3}{6}  }
 \Big)  ,\\
&&\lim_{ \xi_{\rm s} \to \infty} \lim_{\eps \to 0} \Big( \frac{\xi_{\rm s}^2}{2\pi} \Big)^2 
{\rm Var} \big( {W}^\eps_{\rm s} ( z,\br,\bxi) \big) =
\Big( \frac{r_0^2}{ r_0^2 +\frac{\gamma z^3}{6} } \Big)^2
\exp \Big(  - \frac{2 | \br |^2}
{r_0^2  +\frac{\gamma z^3}{6}  }
 \Big)  ,
\end{eqnarray*}
and also
$$
\lim_{ \xi_{\rm s} \to \infty}  \lim_{\eps \to 0}{C}^\eps_{\rm s}(z,\br,\bxi)   =1,
$$
for $r_{\rm s}=0$.
These results are consistent with formulas (\ref{eq:corfieldplane}-\ref{eq:corintplane}) (with $\bq={\bf 0}$) 
and the fact that
$$
\big|  {u} \big(\frac{z}{\eps},\frac{\br}{\eps}\big) \big|^2 = \frac{1}{(2\pi)^2}  \int W^\eps(z,\br,\bxi') d\bxi'
=\lim_{\xi_{\rm s} \to \infty} \frac{\xi_{\rm s}^2}{2\pi} {W}^\eps_{\rm s}(z,\br,\bxi)\mid_{r_{\rm s}=0} .
$$
This shows that the limits $\xi_{\rm s} \to \infty$ and $\eps \to 0$ are exchangeable.

\section{Conclusions}
In this paper we have considered the white-noise paraxial wave model and computed 
the second and fourth-order moments of the field.
In the regime in which the correlation length of the medium is smaller than the initial
beam width, the moments exhibit a multi-scale behavior  with components
varying at these two scales. Our novel  characterization of the solution
of the fourth-order moment equation allows us to solve important questions:
in this paper we have proved that the fourth-order centered moment of the
field satisfies the Gaussian summation rule, 
we analyzed the correlation function of the intensity distribution, and 
we have computed the variance of the smoothed Wigner transform of the transmitted field.
In particular we have characterized quantitatively the amount of smoothing necessary
to get a statistically stable smoothed Wigner transform.
We believe   that our main result can find many other applications, for instance for the 
stability of time-reversal experiments \cite{blomgren,PRS04}
or the stability of correlation-based imaging techniques
in the paraxial regime \cite{dehoopsolna09,dehoopgarnier13}.

\section*{Acknowledgements}
This work is partly supported by  AFOSR grant  \# FA9550-11-1-0176.

\appendix

\section{The Scintillation Regime for the Wave Equation}
\label{app:a}%
In Section \ref{sec:regime} we  address a scaling regime which can be considered
as a particular case of the paraxial white-noise regime:  the scintillation
regime.  This corresponds to a situation in which the relative intensity fluctuations are
of order one and it is an important regime to capture from the physical viewpoint.
We explain in this appendix the conditions for the  validity of this regime
in the context of the wave equation (\ref{eq:helm}).

Let $\sigma$ be the standard deviation of the fluctuations of the index of refraction
$n$ in (\ref{eq:helm}).   Moreover, let
$l_{\rm c}$ be the correlation length of the fluctuations of the index of refraction, $\lambda_0$ 
be the carrier wavelength (equal to $2\pi /k_0$), $L$ be the typical propagation distance, and $r_0$
be the radius of the initial transverse beam/source.
In this framework the variance $C({\bf 0})$ of the Brownian field in the It\^o-Schr\"odinger equation
(\ref{eq:model}) is of order $\sigma^2 l_{\rm c}$ and the transverse scale of variation of the covariance function
$C(\bx)$ in (\ref{defB}) is of order $l_{\rm c}$.

We next  discuss the scintillation scaling regime in more detail.   
First,  we consider the primary scaling that leads to
the canonical white-noise Schr\"odinger equation (\ref{eq:model}), which corresponds 
to zooming in on a high-frequency beam that propagates
over a distance that is large relative to the medium correlation length, which is itself large
relative to the wavelength.
Moreover, the medium fluctuations are relatively small. 
Explicitly, we assume the primary scaling  when
\begin{eqnarray*}
   \frac{l_{\rm c}}{r_0} \sim 1  \, ,   \quad\quad
      \frac{l_{\rm c}}{L}  \sim   \theta\, ,    \quad \quad
     \frac{l_{\rm c}}{\lambda_0}     \sim    \theta^{-1} \,   ,     \quad  \quad
     \sigma^2 \sim \theta^3 \,  , 
\end{eqnarray*}
where $\theta$ is a small dimensionless parameter.
We introduce dimensionless coordinates by:
\begin{eqnarray*}
\bx = { l_{\rm c} }  { \bx' }  , \quad  \quad  
z =  { l_{\rm c} } z', \quad\quad  
k_0 = \frac{k_0' }{  l_{\rm c} \theta }, \quad \quad  
 \nu ( l_{\rm c} z',  l_{\rm c} \bx') = \theta^{3/2}  \nu' \left(z' ,  \bx'  \right)   .
\end{eqnarray*}
Then dropping  ``primes''  we find that in  dimensionless coordinates
the Helmholtz equation reads
$$
(\partial_z^2+\Delta_\bx) v^\theta + \frac{k_0^2}{\theta^2} \Big(1 
+ \theta^{3/2} \nu (z,\bx) \Big) {v}^\theta= 0 .
$$  
We look for the behavior of the slowly-varying envelope $u^\theta$
for long propagation distances of the order of $\theta^{-1}$:
$$
 {v}^\theta \Big(\frac{z}{\theta},\bx \Big)  = \exp\Big( i \frac{k_0 z}{\theta^2} \Big)
u^\theta ( z, \bx  ) 
$$
that satisfies (by the chain rule)
$$
{\theta^2} \partial_{z}^2 u^\theta
+ 
\left( 2 i k_0 \partial_z 
u^\theta+ \Delta_\bx u^\theta +  \frac{k_0^2}{\theta^{1/2}}  \nu\Big(  \frac{z}{\theta} , \bx\Big) 
 u^\theta \right)=0  .
$$
Heuristically, when $\theta \ll 1$ the backscattering term ${\theta^2} \partial_{z}^2  u^\theta$
can be neglected and we obtain a Schr\"odinger-type equation 
in which the potential fluctuates in $z$ on the  
scale $\theta$  and is of amplitude
$ \theta^{-1/2}$. 
This diffusion approximation scaling gives the Brownian field
and the model (\ref{eq:model}):
$$
2 i k_0 d
u + \Delta_\bx u \, dz +   k_0^2 
 u\circ dB(z,\bx) .
$$
This heuristic derivation can be made rigorous as shown in \cite{garniers0,garniers1,garniers2}.

In Section  \ref{sec:regime} we address the subsequent scaling regime in which the correlation 
length of the medium $l_{\rm c}$ is smaller than the initial beam radius $r_0$.
Moreover, the medium fluctuations are relatively weak, and the beam propagates deep
into the medium. We  then get the modified scaling picture
\begin{equation}
\label{eq:scaling app}
    \frac{l_{\rm c}}{r_0} \sim  \eps  \, ,  \quad\quad
     \frac{l_{\rm c}}{L}  \sim   \theta \eps \, ,   \quad\quad
     \frac{l_{\rm c}}{\lambda_0} \sim \theta^{-1}   \,   ,    \quad\quad
     \sigma^2 \sim \theta^3 \eps \,      ,
\end{equation}
and we assume $  \theta \ll \eps \ll 1$.
This means that the paraxial white-noise limit $\theta \to 0$ is taken first, 
and we find
$$
2ik_0 d {u}^\eps      
    +\Delta_{\bx}   {u}^\eps \, dz
   +  k_0^2   {u}^\eps   \circ  d{B}^\eps(z,\bx) 
 =0 , 
$$
where the radius $r_0$ of the initial condition is of  order $\eps^{-1}$,  
the variance $C^\eps({\bf 0})$ of the Brownian field $B^\eps$ is of order $\eps$,
and the propagation distance $L$ is of order $\eps^{-1}$.  
Then the limit  $\eps\to 0$ is applied, corresponding to the scintillation regime. 
In the regime (\ref{eq:scaling app}) the effective strength $k_0^2 C^\eps({\bf 0}) L$ of the 
Brownian field is of order one since $  \sigma^2 l_{\rm c} L/\lambda_0^2   \sim 1$.
Moreover,  $L \lambda_0/ r_0^2$ is of order $\eps$.
That is,   the typical propagation distance is smaller than the Rayleigh length of the initial beam.
Here the Rayleigh length corresponds to the distance when
the transverse radius of the beam has roughly  doubled by diffraction
in the homogeneous medium case and it  is given by $r_0^2/\lambda_0$.
Indeed,  it is seen  in Section \ref{sec:regime} that the propagation distance at which relevant phenomena arise 
in the random case is of the order
of $r_0 l_{\rm c} / \lambda_0$,  which is smaller than the Rayleigh distance $r_0^2/\lambda_0$.

\section{Proof of Proposition \ref{prop:sci1}}
\label{app:b}%
Let $Z>0$.
For any $z \in [0,Z]$ the linear operator ${\cal L}^\eps_{z}$ is bounded from $L^1(\RR^2\times \RR^2 \times \RR^2\times \RR^2) $
into itself and (as in Lemma \ref{lem:normLz})
$$
\sup_{z\leq Z}   \|   {\cal L}^\eps_z    \|_{L^1 \to L^1}  \leq 2 k_0^2 C({\bf 0}) , 
$$
uniformly in $\eps$.
We denote
\begin{eqnarray}
\label{def:Reps}
 R^\eps (z,\bxi_1,\bxi_2, \bzeta_1,\bzeta_2 )  &=&
 \tilde{\mu}^\eps (z,\bxi_1,\bxi_2, \bzeta_1, \bzeta_2 )  -  N^\eps (z,\bxi_1,\bxi_2, \bzeta_1, \bzeta_2 )\\
\nonumber
 N^\eps (z,\bxi_1,\bxi_2, \bzeta_1, \bzeta_2 ) &=&
K(z)
\phi^\eps ( \bxi_1 )
\phi^\eps ( \bxi_2 )
\phi^\eps ( \bzeta_1 )
\phi^\eps ( \bzeta_2 ) \\
\nonumber
&& 
 + \phi^\eps \big( \frac{\bxi_1-\bxi_2}{\sqrt{2}}\big)\phi^\eps ( \bzeta_1 ) \phi^\eps ( \bzeta_2 )
\tilde{A}\big(z, \frac{\bxi_2+\bxi_1}{2} ,\frac{\bzeta_2+ \bzeta_1}{\eps} \big)  
\\
\nonumber
&&+
\phi^\eps \big( \frac{\bxi_1+\bxi_2}{\sqrt{2}}\big)
\phi^\eps ( \bzeta_1 )\phi^\eps ( \bzeta_2 )
\tilde{A} \big(z, \frac{\bxi_2-\bxi_1}{2} ,\frac{\bzeta_2- \bzeta_1}{\eps} \big) \\
\nonumber
&&  +
\phi^\eps \big( \frac{\bxi_1-\bzeta_2}{\sqrt{2}}\big)
\phi^\eps ( \bzeta_1 )  \phi^\eps ( \bxi_2 )
\tilde{A}\big(z, \frac{\bzeta_2+\bxi_1}{2} ,\frac{\bxi_2+ \bzeta_1}{\eps}  \big) 
\\
\nonumber
&& +
\phi^\eps \big( \frac{\bxi_1+\bzeta_2}{\sqrt{2}}\big)
\phi^\eps ( \bzeta_1 ) \phi^\eps ( \bxi_2 )
\tilde{A}\big(z, \frac{\bzeta_2-\bxi_1}{2} ,\frac{\bxi_2-\bzeta_1}{\eps} \big) \\
\nonumber
&&   + \phi^\eps ( \bzeta_1 )\phi^\eps (\bzeta_2) 
\tilde{B} \big( z, \frac{\bxi_2+\bxi_1}{2}, \frac{\bxi_2-\bxi_1}{2} , \frac{\bzeta_1}{\eps}, \frac{\bzeta_2}{\eps}\big) \\
&&   
+
  \phi^\eps ( \bzeta_1 ) \phi^\eps (\bxi_2)
\tilde{B} \big( z, \frac{\bzeta_2+\bxi_1}{2}, \frac{\bzeta_2-\bxi_1}{2},\frac{\bzeta_1}{\eps}, \frac{\bxi_2}{\eps}   \big) .
\label{def:Neps}
\end{eqnarray}

Here (using the definitions (\ref{def:K}) and (\ref{def:A})):\\
- The function $K(z) = (2\pi)^8 \exp (- \frac{k_0^2}{2} C({\bf 0}) z ) $  is the solution of the equation
$$
\frac{\partial K}{\partial z}= \frac{k_0^2}{4( 2\pi)^2} \int \hat{C}(\bk) \big[ -2 K \big]  d\bk  ,
$$
starting from $K(z=0)  =(2\pi)^8$.\\
- The function 
$$
\tilde{A}(z,\bxi,\bzeta) = K(z) A(z,\bxi,\bzeta)
$$
 is the solution of the equation (in which $\bzeta$ is frozen)
$$
\frac{\partial \tilde{A}}{\partial z} =
\frac{k_0^2}{4( 2\pi)^2} \int \hat{C}(\bk) \Big[ \tilde{A}(z,\bxi-\bk,\bzeta) e^{ \frac{i z}{k_0} \bk \cdot \bzeta} - 2 \tilde{A} (z,\bxi,\bzeta) \Big] d\bk
+ \frac{k_0^2}{8(2\pi)^2} \hat{C} (\bxi) K(z)  e^{ i \frac{ z}{k_0} \bxi \cdot \bzeta} ,
$$
starting from $\tilde{A}(z=0,\bxi,\bzeta)  =0$. By Gronwall's inequality 
$\| \tilde{A}(z,\cdot,\bzeta) \|_{L^1}$
is bounded by
\begin{equation}
\label{eq:boundtildeA}
\| \tilde{A}(z,\cdot,\bzeta) \|_{L^1(\RR^2)} \leq (2\pi)^8 
\frac{k_0^2 C({\bf 0}) z}{8} \exp \big( - \frac{k_0^2 C({\bf 0}) z}{4} \big)  ,
\end{equation}
so that it is bounded uniformly in $ \bzeta \in \RR^2$, $z \in [0,Z]$ by
\begin{equation}
\label{eq:Abound}
\sup_{z \in [0,Z],\bzeta\in \RR^2}
\| \tilde{A}(z,\cdot,\bzeta) \|_{L^1(\RR^2)} \leq \frac{(2\pi)^8}{2}
\sup_{z \in [0,Z]}  \frac{k_0^2 C({\bf 0}) z}{4} \exp \big( - \frac{k_0^2 C({\bf 0}) z}{4} \big) 
\leq \frac{(2\pi)^8}{2e} .
\end{equation}
- 
The function 
$$
\tilde{B}(z,\balpha,\bbeta,\bzeta_1,\bzeta_2) = K(z) 
A \big( z, \balpha,  \bzeta_2  +\bzeta_1\big)
A \big( z, \bbeta,  \bzeta_2   -\bzeta_1\big)
$$
is the solution of the equation (in which $\bzeta_1$ and $\bzeta_2$ are frozen):
\begin{eqnarray*}
\frac{\partial \tilde{B}}{\partial z}  &=& 
\frac{k_0^2}{4( 2\pi)^2} \int \hat{C}(\bk) \Big[ 
\tilde{B}(z,\balpha-\bk,\bbeta,\bzeta_1,\bzeta_2) 
e^{i \frac{z}{k_0} \bk \cdot (\bzeta_2+\bzeta_1)} 
\\
&&\hspace*{-0.2in}+
\tilde{B}(z,\balpha,\bbeta-\bk,\bzeta_1,\bzeta_2) 
e^{ i \frac{z}{k_0} \bk \cdot (\bzeta_2-\bzeta_1)} 
- 2 \tilde{B} (z,\balpha ,\bbeta,\bzeta_1,\bzeta_2) \Big] d\bk\\
&&\hspace*{-0.2in}+ \frac{k_0^2}{8(2\pi)^2}  \Big[ 
\hat{C} (\balpha) \tilde{A}(z,\bbeta,\bzeta_2-\bzeta_1) e^{ i \frac{z}{k_0} \balpha \cdot (\bzeta_2+\bzeta_1)} 
+
\hat{C} (\bbeta) \tilde{A}(z,\balpha,\bzeta_2+\bzeta_1) e^{ i \frac{z}{k_0} \bbeta \cdot (\bzeta_2-\bzeta_1)} 
\Big],
\end{eqnarray*}
starting from $\tilde{B}(z=0,\balpha,\bbeta,\bzeta_1,\bzeta_2)  =0$. 
From (\ref{eq:boundtildeA})
$\| \tilde{B}(z,\cdot,\cdot,\bzeta_1,\bzeta_2) \|_{L^1}$
is bounded uniformly in $ \bzeta_1,\bzeta_2 \in \RR^2$, $z \in [0,Z]$ by
$$
\sup_{z \in [0,Z],\bzeta_1,\bzeta_2\in \RR^2}
\| \tilde{B}(z,\cdot,\cdot,\bzeta_1,\bzeta_2) \|_{L^1(\RR^2 \times \RR^2)} \leq  (2\pi)^8 
\Big( \frac{k_0^2 C({\bf 0}) Z}{8}  \Big)^2.
$$

The strategy is to show that the remainder $R^\eps$ in (\ref{def:Reps})  belongs to $L^1$ and that its 
$L^1$-norm goes to zero as $\eps \to 0$ uniformly in  $z\in [0,Z]$.
To this effect  we will first show that $R^\eps $ satisfies an equation with zero initial condition
and with a source term (Lemma \ref{lem:lem1R}),  then
that the source term is small in $L^1$-norm (Lemma \ref{lem:lem2R}),
 and we finally  get the desired result by a Gronwall-type argument (Lemma \ref{lem:lem3R}).

\begin{lemma}
\label{lem:lem1R}
$R^\eps$ satisfies 
\begin{equation}
 \label{eq:tildeReps2}
 \frac{\partial R^\eps }{\partial z}
 (z,\bxi_1,\bxi_2, \bzeta_1,\bzeta_2 )= \big[ {\cal L}^\eps_z R^\eps \big] (z,\bxi_1,\bxi_2, \bzeta_1,\bzeta_2 )
 + S^\eps(z,\bxi_1,\bxi_2, \bzeta_1, \bzeta_2 ),
\end{equation}
 starting from $R^\eps (z=0,\bxi_1,\bxi_2, \bzeta_1,\bzeta_2 )=0$, 
 with the source term $S^\eps$ given by
\begin{equation}
S^\eps (z,\bxi_1,\bxi_2,\bzeta_1, \bzeta_2 )= S_{1}^\eps(z,\bxi_1,\bxi_2, \bzeta_1, \bzeta_2 )+S_{2}^\eps(z,\bxi_1,\bxi_2, \bzeta_1,\bzeta_2 ) ,
\label{def:Repsproof3}
\end{equation}
with
\begin{eqnarray}
S^\eps_{1}(z,\bxi_1,\bxi_2, \bzeta_1, \bzeta_2 )&=& - \frac{\partial N^\eps}{\partial z}(z,\bxi_1,\bxi_2,\bzeta_1, \bzeta_2 ),\\
S_{2}^\eps(z,\bxi_1,\bxi_2, \bzeta_1, \bzeta_2 ) &=& \big[ {\cal L}^\eps_z  N^\eps  \big](z,\bxi_1,\bxi_2,\bzeta_1, \bzeta_2 ) .
\end{eqnarray}
\end{lemma}

\begin{proof}
 By taking the $z$-derivative of $R^\eps $, and using $R^\eps = \tilde{\mu}^\eps -
 N^\eps $,
 we find that
\begin{eqnarray*}
\frac{\partial R^\eps }{\partial z}
&=&
\frac{\partial \tilde{\mu}^\eps}{\partial z} - 
\frac{\partial N^\eps }{\partial z} 
\\ &=& 
 \big[ {\cal L}^\eps_z  \tilde{\mu}^\eps \big]- 
\frac{\partial N^\eps }{\partial z}  
\\ &=&
  \big[ {\cal L}^\eps_z   R^\eps \big]
+
\big[ {\cal L}^\eps_z N^\eps \big]
- 
\frac{\partial N^\eps}{\partial z}  ,
\end{eqnarray*}
 which gives the desired result.
 \qed
\end{proof}

\begin{lemma}
\label{lem:lem2R}%
For any $Z>0$ we have
\begin{equation}
\label{eq:lem2R}
\sup_{z \in [0,Z]} \Big\| \int_0^z 
{S}^\eps  (z',\cdot,\cdot,\cdot,\cdot) dz'\Big\|_{L^1(\RR^2 \times \RR^2 \times \RR^2\times \RR^2)}  \stackrel{\eps \to 0}{\longrightarrow} 0  .
\end{equation}
\end{lemma}

\begin{proof}
There are three types of contributions to  $S_{1}^\eps$, the one  that involves $K$,
the ones that involve $\tilde{A}$, and the ones that involve $\tilde{B}$.
We decompose  $S_{1}^\eps$ into three terms corresponding to these three
contributions.
\begin{eqnarray*}
S_{1}^\eps  =  S_{K}^\eps +  S_{A}^\eps + S_{B}^\eps  .
 \end{eqnarray*}
From (\ref{def:Neps}) and the differential equations satisfied by $K$, $\tilde{A}$, and $\tilde{B}$,
 the components of  $S^\eps_{ 1}$ are  given explicitly by
 \begin{eqnarray}
 &&
S_{K}^\eps(z,\bxi_1,\bxi_2, \bzeta_1,\bzeta_2 )=  \frac{ k_0^2}{4(2\pi)^2}  \int \hat{C}(\bk) \bigg\{ 
 2 K \phi^\eps (\bxi_1) \phi^\eps(\bxi_2) \phi^\eps (\bzeta_1)  \phi^\eps (\bzeta_2)    
\bigg\} d\bk ,
\label{def:SepsK}
\end{eqnarray}

\begin{eqnarray}
\nonumber
&&
S_{A}^\eps(z,\bxi_1,\bxi_2, \bzeta_1,\bzeta_2 )= - \frac{ k_0^2}{4(2\pi)^2} \phi^\eps (\bzeta_1) \int \hat{C}(\bk) \bigg\{    \\ && 
  \phi^\eps \big(\frac{\bxi_1-\bxi_2}{\sqrt{2}} \big) \phi^\eps(\bzeta_2)
\Big[ \tilde{A}\big( \frac{\bxi_2+\bxi_1}{2} - \bk,\frac{\bzeta_2+\bzeta_1}{\eps}\big)
e^{ i\frac{z}{\eps k_0} \bk \cdot ( \bzeta_2  +\bzeta_1)} - 2  \tilde{A}\big( \frac{\bxi_2+\bxi_1}{2},\frac{\bzeta_2+\bzeta_1}{\eps}\big)
\Big]   \nonumber  \\
\nonumber
&& + \phi^\eps \big(\frac{\bxi_1+\bxi_2}{\sqrt{2}} \big) \phi^\eps(\bzeta_2)
\Big[ \tilde{A}\big( \frac{\bxi_2-\bxi_1}{2} - \bk,\frac{\bzeta_2-\bzeta_1}{\eps}\big)
e^{ i\frac{z}{\eps k_0} \bk \cdot ( \bzeta_2  -\bzeta_1)} - 2  \tilde{A}\big( \frac{\bxi_2-\bxi_1}{2},\frac{\bzeta_2-\bzeta_1}{\eps} \big)
 \Big] \\
\nonumber
&& + \phi^\eps \big(\frac{\bxi_1-\bzeta_2}{\sqrt{2}} \big) \phi^\eps(\bxi_2)
\Big[ \tilde{A}\big( \frac{\bzeta_2+\bxi_1}{2} - \bk,\frac{\bxi_2+\bzeta_1}{\eps} \big)
e^{ i\frac{z}{\eps k_0} \bk \cdot ( \bxi_2 +\bzeta_1)} - 2  \tilde{A}\big( \frac{\bzeta_2+\bxi_1}{2},\frac{\bxi_2+\bzeta_1}{\eps} \big)
\Big] \\
\nonumber
&& + \phi^\eps \big(\frac{\bxi_1+\bzeta_2}{\sqrt{2}} \big) \phi^\eps(\bxi_2)
\Big[ \tilde{A}\big( \frac{\bzeta_2-\bxi_1}{2} - \bk,\frac{\bxi_2-\bzeta_1}{\eps} \big)
e^{ i\frac{z}{\eps k_0} \bk \cdot ( \bxi_2  -\bzeta_1)} - 2  \tilde{A}\big( \frac{\bzeta_2-\bxi_1}{2},\frac{\bxi_2-\bzeta_1}{\eps} \big)
\Big]   \bigg\} d\bk   \nonumber \\
\nonumber
&&   - \frac{k_0^2}{8(2\pi)^2}  
\phi^\eps (\bzeta_1)    \bigg\{  
 \phi^\eps \big(\frac{\bxi_1-\bxi_2}{\sqrt{2}} \big) \phi^\eps(\bzeta_2)
K \hat{C}\big( \frac{\bxi_2+\bxi_1}{2}\big) 
e^{ i\frac{z}{\eps k_0} \frac{\bxi_2+\bxi_1}{2}  \cdot (\bzeta_2+\bzeta_1 )} 
 \\
\nonumber
&& \hspace*{1.0in} +   \phi^\eps \big(\frac{\bxi_1+\bxi_2}{\sqrt{2}} \big) \phi^\eps(\bzeta_2)
K \hat{C}\big( \frac{\bxi_2-\bxi_1}{2}\big) 
e^{ i\frac{z}{\eps k_0} \frac{\bxi_2-\bxi_1}{2}  \cdot ( \bzeta_2-\bzeta_1  )} 
 \\
\nonumber
&& \hspace*{1.0in} +  \phi^\eps \big(\frac{\bzeta_2-\bxi_1}{\sqrt{2}} \big) \phi^\eps(\bxi_2)
K \hat{C}\big( \frac{\bzeta_2+\bxi_1}{2}\big) 
e^{ i\frac{z}{\eps k_0} \frac{\bzeta_2+\bxi_1}{2}  \cdot (\bxi_2 +\bzeta_1)} 
 \\
 && \hspace*{1.0in} +   \phi^\eps \big(\frac{\bzeta_2+\bxi_1}{\sqrt{2}} \big) \phi^\eps(\bxi_2)
K \hat{C}\big( \frac{\bzeta_2-\bxi_1}{2}\big) 
e^{ i\frac{z}{\eps k_0} \frac{\bzeta_2-\bxi_1}{2}  \cdot ( \bxi_2-\bzeta_1 )} 
 \Big] 
\bigg\}  ,
\label{def:SepsA}
\end{eqnarray}
 
 \begin{eqnarray}
\nonumber
&&
S_{B}^\eps(z,\bxi_1,\bxi_2, \bzeta_1,\bzeta_2 )= - \frac{ k_0^2}{4(2\pi)^2} \phi^\eps (\bzeta_1) \int \hat{C}(\bk) \bigg\{ 
  \\
\nonumber
&&  \phi^\eps(\bzeta_2)\Big[ \tilde{B}\big( \frac{\bxi_2+\bxi_1}{2} - \bk,\frac{\bxi_2-\bxi_1}{2},\frac{\bzeta_1}{\eps},\frac{\bzeta_2}{\eps} \big)
e^{ i\frac{z}{\eps k_0} \bk \cdot ( \bzeta_2 +\bzeta_1)}  \\
\nonumber
&& \quad +
 \tilde{B}\big( \frac{\bxi_2+\bxi_1}{2},\frac{\bxi_2-\bxi_1}{2} - \bk,\frac{\bzeta_1}{\eps},\frac{\bzeta_2}{\eps} \big)
e^{ i\frac{z}{\eps k_0} \bk \cdot ( \bzeta_2 -\bzeta_1)} 
-2  \tilde{B}\big( \frac{\bxi_2+\bxi_1}{2}  ,\frac{\bxi_2-\bxi_1}{2},\frac{\bzeta_1}{\eps},\frac{\bzeta_2}{\eps} \big)
\Big]\\
\nonumber
&&
+\phi^\eps(\bxi_2) \Big[ \tilde{B} \big( \frac{\bzeta_2+\bxi_1}{2} - \bk, \frac{\bzeta_2-\bxi_1}{2} ,\frac{\bzeta_1}{\eps},\frac{\bxi_2}{\eps}\big)
e^{ i\frac{z}{\eps k_0} \bk \cdot ( \bxi_2+\bzeta_1 )} \\
\nonumber
&& \quad + \tilde{B} \big( \frac{\bzeta_2+\bxi_1}{2} , \frac{\bzeta_2-\bxi_1}{2}-\bk , \frac{\bzeta_1}{\eps},\frac{\bxi_2}{\eps}\big)
e^{ i\frac{z}{\eps k_0} \bk \cdot ( \bxi_2-\bzeta_1   )} 
- 2\tilde{B} \big( \frac{\bzeta_2+\bxi_1}{2} , \frac{\bzeta_2-\bxi_1}{2} , \frac{\bzeta_1}{\eps},\frac{\bxi_2}{\eps}\big)
\Big] \bigg\} d\bk\\
\nonumber
&&   - \frac{k_0^2}{8(2\pi)^2} \phi^\eps (\bzeta_1)
\bigg\{  
  \phi^\eps(\bzeta_2) \Big[
 \hat{C}\big( \frac{\bxi_2+\bxi_1}{2}\big) \tilde{A}\big( \frac{\bxi_2-\bxi_1}{2}, \frac{\bzeta_2-\bzeta_1}{\eps} \big)
e^{ i\frac{z}{\eps k_0} \frac{\bxi_2+\bxi_1}{2}  \cdot (\bzeta_2 +\bzeta_1)} \\
\nonumber
&&\hspace*{1.3in}  \quad +
\hat{C}\big( \frac{\bxi_2-\bxi_1}{2}\big) \tilde{A}\big( \frac{\bxi_2+\bxi_1}{2}, \frac{\bzeta_2+\bzeta_1}{\eps} \big)
e^{ i\frac{z}{\eps k_0} \frac{\bxi_2-\bxi_1}{2}  \cdot ( \bzeta_2-\bzeta_1  )} 
\Big] \\
\nonumber
&& \hspace*{0.95in}+ \phi^\eps(\bxi_2) \Big[
 \hat{C}\big( \frac{\bzeta_2+\bxi_1}{2}\big) \tilde{A}\big( \frac{\bzeta_2-\bxi_1}{2}, \frac{\bxi_2-\bzeta_1}{\eps} \big)
e^{ i\frac{z}{\eps k_0} \frac{\bzeta_2+\bxi_1}{2}  \cdot ( \bxi_2  +\bzeta_1)} \\
&& \hspace*{1.3in}\quad +
\hat{C}\big( \frac{\bzeta_2-\bxi_1}{2}\big) \tilde{A}\big( \frac{\bzeta_2+\bxi_1}{2}, \frac{\bxi_2+\bzeta_1}{\eps} \big)
e^{ i\frac{z}{\eps k_0} \frac{\bzeta_2-\bxi_1}{2}  \cdot ( \bxi_2  -\bzeta_1)} 
\Big] 
\bigg\}  .
\label{def:SepsB}
\end{eqnarray}
 
$S_{2}^\eps$  is given by ${\cal L}_z^\eps N^\eps $,
 with $N^\eps $ given by (\ref{def:Neps}).
 Therefore  we can express $S_{ 2}^\eps$ as 
\begin{eqnarray}
\nonumber
S_{2}^\eps (z,\bxi_1,\bxi_2, \bzeta_1,\bzeta_2 )&=&
{\cal L}_{z}^\eps \big[ K(z) \phi^\eps(\bxi_1)\phi^\eps(\bxi_2)\phi^\eps(\bzeta_1)\phi^\eps(\bzeta_2) \big] 
\\
\nonumber
&& 
 + {\cal L}_{z}^\eps \big[ \phi^\eps \big( \frac{\bxi_1-\bxi_2}{\sqrt{2}}\big) \phi^\eps(\bzeta_1)\phi^\eps ( \bzeta_2 )
\tilde{A}\big(z, \frac{\bxi_2+\bxi_1}{2} ,\frac{\bzeta_2+\bzeta_1}{\eps} \big)  \big]
\\
\nonumber
&& 
+
{\cal L}_{z}^\eps \big[\phi^\eps \big( \frac{\bxi_1+\bxi_2}{\sqrt{2}}\big)
\phi^\eps(\bzeta_1) \phi^\eps ( \bzeta_2 )
\tilde{A} \big(z, \frac{\bxi_2-\bxi_1}{2} ,\frac{\bzeta_2-\bzeta_1}{\eps} \big) \big]
\\
\nonumber
&&  +
{\cal L}_{z}^\eps \big[\phi^\eps \big( \frac{\bxi_1-\bzeta_2}{\sqrt{2}}\big)
\phi^\eps ( \bxi_2 ) \phi^\eps(\bzeta_1)
\tilde{A}\big(z, \frac{\bzeta_2+\bxi_1}{2} ,\frac{\bxi_2+\bzeta_1}{\eps} \big) \big]
\\
\nonumber
&& 
+
{\cal L}_{z}^\eps \big[\phi^\eps \big( \frac{\bxi_1+\bzeta_2}{\sqrt{2}}\big)
\phi^\eps ( \bxi_2 ) \phi^\eps(\bzeta_1)
\tilde{A}\big(z, \frac{\bzeta_2-\bxi_1}{2} ,\frac{\bxi_2-\bzeta_1}{\eps}  \big) \big]
\\
\nonumber
&& 
   + {\cal L}_{z}^\eps \big[\phi^\eps (\bzeta_2) \phi^\eps(\bzeta_1)
\tilde{B} \big( z, \frac{\bxi_2+\bxi_1}{2}, \frac{\bxi_2-\bxi_1}{2} ,\frac{\bzeta_1}{\eps}, \frac{\bzeta_2}{\eps}\big) \big]
\\
&& 
+
 {\cal L}_{z}^\eps \big[\phi^\eps (\bxi_2) \phi^\eps(\bzeta_1)
\tilde{B} \big( z, \frac{\bzeta_2+\bxi_1}{2}, \frac{\bzeta_2-\bxi_1}{2},\frac{\bzeta_1}{\eps}, \frac{\bxi_2}{\eps}   \big)\big] .
\label{eq:tildeNeps1}
\end{eqnarray}
It turns out that all the terms in $S_{1}^\eps$ are canceled by terms in $S_{2}^\eps$,
and the last terms of $S_{2}^\eps$ are small, as will be shown below.

Again there are three types of contributions in the expression (\ref{eq:tildeNeps1}) for  $S_{2}^\eps$, the one that involves $K$,
the ones that involve $\tilde{A}$, and the ones that involve $\tilde{B}$.
We will study one contribution for each of these three types and show the desired result for them.\\

Let us examine the contributions of $K(z) \phi^\eps(\bxi_1)\phi^\eps(\bxi_2)\phi^\eps(\bzeta_1)\phi^\eps(\bzeta_2)$ to $S_{2}^\eps$:
\begin{eqnarray}
\nonumber
&& 
{\cal L}_{z}^\eps \big[ K(z) \phi^\eps(\bxi_1)\phi^\eps(\bxi_2)\phi^\eps(\bzeta_1)\phi^\eps(\bzeta_2) \big]
= \frac{k_0^2}{4(2\pi)^2} K(z) \phi^\eps(\bzeta_1) \int  \hat{C}(\bk) \bigg[ -2 \phi^\eps(\bxi_1)\phi^\eps(\bxi_2)\phi^\eps(\bzeta_2)
\\
\nonumber
&&
+
\phi^\eps(\bxi_1-\bk)\phi^\eps(\bxi_2-\bk) \phi^\eps(\bzeta_2) e^{i \frac{z}{\eps k_0} \bk \cdot ( \bzeta_2 + \bzeta_1)}
+
\phi^\eps(\bxi_1-\bk)\phi^\eps(\bzeta_2-\bk) \phi^\eps(\bxi_2) e^{i \frac{z}{\eps k_0} \bk \cdot ( \bxi_2 + \bzeta_1)}\\
\nonumber
&&+
\phi^\eps(\bxi_1+\bk)\phi^\eps(\bxi_2-\bk) \phi^\eps(\bzeta_2) e^{i \frac{z}{\eps k_0} \bk \cdot (  \bzeta_2  - \bzeta_1)}
+
\phi^\eps(\bxi_1+\bk)\phi^\eps(\bzeta_2-\bk) \phi^\eps(\bxi_2) e^{i \frac{z}{\eps k_0} \bk \cdot (  \bxi_2  - \bzeta_1)}\\
\nonumber
&&
- \phi^\eps (\bxi_1) \phi^\eps(\bxi_2-\bk) \phi^\eps( \bzeta_2-\bk) e^{i \frac{z}{\eps k_0 } \big(\bk \cdot(\bzeta_2+\bxi_2 ) -|\bk|^2\big)} \\
&&
- \phi^\eps (\bxi_1) \phi^\eps(\bxi_2-\bk) \phi^\eps( \bzeta_2+\bk) e^{i \frac{z}{\eps k_0} \big(\bk \cdot(\bzeta_2-\bxi_2 ) +|\bk|^2\big)}
\bigg] d\bk
.
\label{eq:tildeNeps1b}
\end{eqnarray}
The first term  cancels with  the  term  $S_{K}^\eps$.
The second term can be rewritten since
$$
\phi^\eps(\bxi_1-\bk)\phi^\eps(\bxi_2-\bk) = \phi^\eps\Big( \sqrt{2}\big( \bk - \frac{\bxi_1+\bxi_2}{2}\big) \Big)
\phi^\eps\Big( \frac{\bxi_1-\bxi_2}{\sqrt{2}}\Big)   ,
$$
and therefore, up to a negligible term in $L^1(\RR^2 \times \RR^2 \times \RR^2\times \RR^2)$,
\begin{eqnarray}
\nonumber
&&\int  \hat{C}(\bk)\phi^\eps(\bxi_1-\bk)\phi^\eps(\bxi_2-\bk) \phi^\eps(\bzeta_1) \phi^\eps(\bzeta_2) e^{i \frac{z}{\eps k_0} \bk \cdot ( \bzeta_2 + \bzeta_1)}
d\bk \\
&&= 
\frac{1}{2}
\hat{C} \big( \frac{\bxi_1+\bxi_2}{2}\big)\phi^\eps\Big( \frac{\bxi_1-\bxi_2}{\sqrt{2}}\Big)  \phi^\eps(\bzeta_1) \phi^\eps(\bzeta_2) 
e^{i \frac{z}{\eps k_0} \frac{\bxi_1+\bxi_2}{2}  \cdot ( \bzeta_2  + \bzeta_1)} + o(1) ,
\label{eq:tildeNeps1c}
\end{eqnarray}
that cancels with  the first ``source'' term  in $S_{A}^\eps$.
The $o(1)$ characterization follows from the following arguments:
\begin{eqnarray*}
&&
\iint \Big| \int  \hat{C}(\bk)\phi^\eps(\bxi_1-\bk)\phi^\eps(\bxi_2-\bk) \phi^\eps(\bzeta_1) \phi^\eps(\bzeta_2) e^{i \frac{z}{\eps k_0} \bk \cdot ( \bzeta_2 + \bzeta_1)}
d\bk \\
&& \quad -
\frac{1}{2}
\hat{C} \big( \frac{\bxi_1+\bxi_2}{2}\big)\phi^\eps\Big( \frac{\bxi_1-\bxi_2}{\sqrt{2}}\Big)  \phi^\eps(\bzeta_1) \phi^\eps(\bzeta_2) 
e^{i \frac{z}{\eps k_0} \frac{\bxi_1+\bxi_2}{2}  \cdot ( \bzeta_2  + \bzeta_1)}  \Big| d\bxi_1 d\bxi_2 d\bzeta_1d\bzeta_2 \\
&& =
\iint \Big| \int  \hat{C}(\bk) \phi^\eps\Big( \sqrt{2}\big( \bk - \frac{\bxi_1+\bxi_2}{2}\big) \Big)
  e^{i \frac{z}{\eps k_0} \bk \cdot ( \bzeta_2 + \bzeta_1)}
d\bk \\
&& \quad -
\frac{1}{2}
\hat{C} \big( \frac{\bxi_1+\bxi_2}{2}\big) 
e^{i \frac{z}{\eps k_0} \frac{\bxi_1+\bxi_2}{2}  \cdot ( \bzeta_2  + \bzeta_1)}  \Big|
 \phi^\eps\Big( \frac{\bxi_1-\bxi_2}{\sqrt{2}}\Big) 
  \phi^\eps (\bzeta_1) \phi^\eps(\bzeta_2) d\bxi_1 d\bxi_2 d\bzeta_1d\bzeta_2 \\
  &&= 
\iint \Big| \int  \hat{C}(\bk)  \phi^\eps\big( \sqrt{2}( \bk - \bxi ) \big)
  e^{i \frac{z}{ k_0} \bk \cdot ( \bzeta_2 + \bzeta_1)}
d\bk \\
&& \quad -
\frac{1}{2}
\hat{C} ( \bxi ) 
e^{i \frac{z}{k_0} \bxi \cdot ( \bzeta_2  + \bzeta_1)}  \Big|
 \phi^1\Big( \frac{\bzeta}{\sqrt{2}}\Big) 
  \phi^1 (\bzeta_1) \phi^1(\bzeta_2) d\bxi  d\bzeta d\bzeta_1d\bzeta_2 \\
  && = 2
\iint \Big| \int  \hat{C}(\bk) \phi^\eps\big( \sqrt{2}( \bk - \bxi ) \big) 
  e^{i \frac{z}{ k_0} (\bk -\bxi)\cdot ( \bzeta_2' + \bzeta_1')}
d\bk   \\
  &&
  \quad  -
\frac{1}{2}
\hat{C} ( \bxi ) 
 \Big|
  \phi^1 \Big( \frac{\bzeta_1'+\bzeta_2'}{\sqrt{2}}\Big) \phi^1\Big( \frac{\bzeta_1'-\bzeta_2'}{\sqrt{2}}\Big) d\bxi  d\bzeta_1'd\bzeta_2' \\
   && = 2
\iint \Big| \int \big( \hat{C}(\bxi+\eps \bk)
  e^{i \eps \sqrt{2}  \frac{z}{ k_0}  {\bk}\cdot \bzeta'}- \hat{C} ( \bxi ) \big)  \phi^1\big( \sqrt{2} \bk\big)
d\bk  
 \Big|
  \phi^1  ( \bzeta' )  d\bxi  d\bzeta' \\
  && \leq 
2 \iint  \big| \hat{C}(\bxi+\eps \bk) - \hat{C} ( \bxi ) \big|   \phi^1\big(\sqrt{2} \bk\big)
  \phi^1 (\bzeta') d\bk   d\bxi  d\bzeta' \\
  && \quad + 2
\iint    \big|  
  e^{i \eps \sqrt{2}  \frac{z}{ k_0} \bk\cdot \bzeta'} - 1\big| 
  \hat{C} ( \bxi )
   \phi^1\big(\sqrt{2} \bk\big)
  \phi^1 (\bzeta') d\bk   d\bxi  d\bzeta' ,
\end{eqnarray*}
where 
$$
\phi^1 (\bxi)  = \frac{r_0^2}{2\pi} \exp \Big(- \frac{r_0^2 |\bxi|^2}{2} \Big)  ,
$$
whose $L^1$-norm is one.
The first term in the right-hand side goes to zero as $\eps \to 0$ by Lebesgue's  dominated convergence theorem
(since $C$ is in $L^1$, $\hat{C}$ is continuous, and since $C({\bf 0})<\infty$, the nonnegative-valued function
$\hat{C}$ is in $L^1$).
The second term can be bounded by
$$
2 \iint    \big|  
  e^{i \eps \sqrt{2}  \frac{z}{ k_0} \bk\cdot \bzeta'} - 1\big| 
  \hat{C} ( \bxi )
   \phi^1\big(\sqrt{2} \bk\big)
  \phi^1 (\bzeta') d\bk   d\bxi  d\bzeta' \leq 
  \eps \frac{Z}{k_0} \Big( \int |\bk| \phi^1 (\bk) d\bk \Big)^2 \Big( \int \hat{C}(\bxi) d\bxi \Big) ,
 $$
which shows that it also goes to zero as $\eps \to 0$ and which justifies the $o(1)$ in (\ref{eq:tildeNeps1c}).
The third, fourth, and fifth terms of the right-hand side of (\ref{eq:tildeNeps1b}) can be dealt with in the same way and cancel the next three ``source''  terms in $S_{A}^\eps$.
The last two terms give negligible contributions in the sense of (\ref{eq:lem2R}).
Indeed, for instance, the sixth term satisfies (using the change of variables $(\bzeta_2,\bxi_2)
\to ( \bzeta= (\bzeta_2-\bk)/\eps,\bxi=(\bxi_2-\bk)/\eps)$): 
\begin{eqnarray*}
&&   \iint \bigg| \int_0^z dz'\int d \bk \hat{C}(\bk) K(z')
  \phi^\eps(\bzeta_1) \phi^\eps(\bxi_1) \phi^\eps(\bxi_2-\bk)
 \phi^\eps( \bzeta_2-\bk) e^{i \frac{z'}{k_0\eps } \big(\bk \cdot(\bzeta_2+\bxi_2 )  -|\bk|^2\big)} 
 \bigg|    \\ &&  \hbox{} \times 
   d\bzeta_1
 d\bzeta_2 d \bxi_1 d \bxi_2   \leq \iint \bigg| \int_0^z dz' \hat{C}(\bk) K(z')
 \phi^1(\bxi) \phi^1( \bzeta ) e^{ i \frac{z'}{k_0} \bk\cdot(\bxi+\bzeta) } 
e^{ i \frac{z'}{k_0\eps } |\bk|^2 } 
 \bigg|
d\bk  d\bzeta d \bxi .
\end{eqnarray*} 
From Lemma \ref{lem:tech1}  this term goes to zero as $\eps \to 0$.\\

Let us examine the contributions of $\phi^\eps(\frac{\bxi_1-\bxi_2}{\sqrt{2}})\phi^\eps(\bzeta_1) \phi^\eps(\bzeta_2) \tilde{A}\big(z,\frac{\bxi_2+\bxi_1}{2},
\frac{\bzeta_2+\bzeta_1}{\eps} \big)$ to $S_{2}^\eps$:
\begin{eqnarray*}
&& 
{\cal L}_{z}^\eps \big[ \phi^\eps(\frac{\bxi_1-\bxi_2}{\sqrt{2}})\phi^\eps(\bzeta_1) \phi^\eps(\bzeta_2) \tilde{A}\big(z,\frac{\bxi_2+\bxi_1}{2},
\frac{\bzeta_2+\bzeta_1}{\eps}  \big) \big]
= \frac{k_0^2}{4(2\pi)^2} \phi^\eps(\bzeta_1) \int  \hat{C}(\bk) \\
&& \times \bigg[ -2 \phi^\eps(\frac{\bxi_1-\bxi_2}{\sqrt{2}}) \phi^\eps(\bzeta_2) \tilde{A}\big(z,\frac{\bxi_2+\bxi_1}{2},
\frac{\bzeta_2+\bzeta_1}{\eps}  \big) \\
&& + \phi^\eps(\frac{\bxi_1-\bxi_2}{\sqrt{2}}) \phi^\eps(\bzeta_2) \tilde{A}\big(z,\frac{\bxi_2+\bxi_1}{2}-\bk,
\frac{\bzeta_2+\bzeta_1}{\eps} \big) e^{i \frac{z}{\eps k_0} \bk \cdot (  \bzeta_2  + \bzeta_1)}\\
&& + \phi^\eps(\frac{\bxi_1-\bxi_2-\bk}{\sqrt{2}}) \phi^\eps(\bzeta_2-\bk) \tilde{A}\big(z,\frac{\bxi_2+\bxi_1-\bk}{2},
\frac{\bzeta_2+\bzeta_1-\bk}{\eps} \big) e^{i \frac{z}{\eps k_0} \bk \cdot ( \bxi_2 + \bzeta_1)}\\
&& + \phi^\eps(\frac{\bxi_1-\bxi_2+2\bk}{\sqrt{2}}) \phi^\eps(\bzeta_2) \tilde{A}\big(z,\frac{\bxi_2+\bxi_1}{2},
\frac{\bzeta_2+\bzeta_1}{\eps} \big) e^{i \frac{z}{\eps k_0} \bk \cdot (  \bzeta_2  - \bzeta_1)}\\
&& + \phi^\eps(\frac{\bxi_1-\bxi_2+\bk}{\sqrt{2}}) \phi^\eps(\bzeta_2-\bk) \tilde{A}\big(z,\frac{\bxi_2+\bxi_1+\bk}{2},
\frac{\bzeta_2+\bzeta_1-\bk}{\eps}  \big) e^{i \frac{z}{\eps k_0} \bk \cdot ( \bxi_2 -\bzeta_1)}\\
&& -  \phi^\eps(\frac{\bxi_1-\bxi_2+\bk}{\sqrt{2}}) \phi^\eps(\bzeta_2-\bk) \tilde{A}\big(z,\frac{\bxi_2+\bxi_1-\bk}{2},
\frac{\bzeta_2+\bzeta_1-\bk}{\eps}  \big) e^{i \frac{z}{\eps k_0}  \big( \bk \cdot (\bzeta_2+\bxi_2) -|\bk|^2\big)}\\
&& -  \phi^\eps(\frac{\bxi_1-\bxi_2+\bk}{\sqrt{2}}) \phi^\eps(\bzeta_2+\bk) \tilde{A}\big(z,\frac{\bxi_2+\bxi_1-\bk}{2},
\frac{\bzeta_2+\bzeta_1+\bk}{\eps} \big) e^{i \frac{z}{\eps k_0}  \big( \bk \cdot  (\bzeta_2-\bxi_2) +|\bk|^2\big)}
\bigg] d\bk  .
\end{eqnarray*}
The first and second terms will be canceled by the corresponding terms in $S_{A}^\eps$.
The fourth term can be rewritten up to a negligible term (in $L^1(\RR^2 \times \RR^2 \times \RR^2 \times \RR^2)$) as
\begin{eqnarray*}
&& \int  \hat{C}(\bk)  
 \phi^\eps(\frac{\bxi_1-\bxi_2+2\bk}{\sqrt{2}}) \phi^\eps(\bzeta_1)\phi^\eps(\bzeta_2) \tilde{A}\big(z,\frac{\bxi_2+\bxi_1}{2},
\frac{\bzeta_2+\bzeta_1}{\eps} \big) e^{i \frac{z}{\eps k_0} \bk \cdot (  \bzeta_2  - \bzeta_1)}
 d\bk\\
&& = \frac{1}{2} \hat{C}\big( \frac{\bxi_2-\bxi_1}{2} \big) \phi^\eps(\bzeta_1) \phi^\eps(\bzeta_2) \tilde{A}\big(z,\frac{\bxi_2+\bxi_1}{2},
\frac{\bzeta_2+\bzeta_1}{\eps}  \big) e^{i \frac{z}{\eps k_0} \frac{\bxi_2-\bxi_1}{2} \cdot ( \bzeta_2-\bzeta_1  )} +o(1) .
\end{eqnarray*}
Therefore the fourth term will be canceled by the corresponding 
``source'' term in $S_{B}^\eps$.
The other terms are negligible in the sense of (\ref{eq:lem2R}).
Indeed, for instance, the third term satisfies (using the change of variables $(\bzeta_1,\bzeta_2,\bxi_1,\bxi_2)
\to (\bxi=\bzeta_1/\eps, \bzeta= (\bzeta_2-\bk)/\eps, \balpha= (\bxi_2+\bxi_1-\bk)/2,\bbeta=(\bxi_1-\bxi_2-\bk)/(\eps \sqrt{2}))$):
\begin{eqnarray*}
&& 
\iint \bigg| \int_0^z dz' \int d\bk \hat{C}(\bk)
\phi^\eps(\frac{\bxi_1-\bxi_2-\bk}{\sqrt{2}})  \phi^\eps(\bzeta_1) \phi^\eps(\bzeta_2-\bk) \\
&& \quad \times \tilde{A}\big(z',\frac{\bxi_2+\bxi_1-\bk}{2},
\frac{\bzeta_2+\bzeta_1-\bk}{\eps} \big) e^{i \frac{z'}{\eps k_0} \bk \cdot ( \bxi_2+ \bzeta_1)}
\bigg| d\bxi_1 d\bxi_2 d\bzeta_1d\bzeta_2\\
&& \leq 
2\iint \bigg| \int_0^z dz' \hat{C}(\bk)
\phi^1(\bbeta) \phi^1(\bxi) \phi^1(\bzeta) \tilde{A}\big(z',\balpha,
\bzeta +\bxi\big)
 e^{i \frac{z'}{k_0} \bk \cdot (\bxi-\frac{\bbeta}{\sqrt{2}})}
 e^{i \frac{z'}{\eps k_0} \bk \cdot \balpha }
\bigg| d\bk d\balpha d\bbeta d\bzeta d\bxi .
\end{eqnarray*}
From Lemma \ref{lem:tech1} this term goes to zero as $\eps \to 0$.\\

Let us examine finally the contributions of $\phi^\eps (\bzeta_1) \phi^\eps (\bzeta_2) 
\tilde{B} \big( z, \frac{\bxi_2+\bxi_1}{2}, \frac{\bxi_2-\bxi_1}{2} ,\frac{\bzeta_1}{\eps}, \frac{\bzeta_2}{\eps}\big) $
 to $S_{2}^\eps$:
\begin{eqnarray*}
&& 
{\cal L}_{z}^\eps \big[\phi^\eps (\bzeta_1)  \phi^\eps (\bzeta_2) 
\tilde{B} \big( z, \frac{\bxi_2+\bxi_1}{2}, \frac{\bxi_2-\bxi_1}{2} ,\frac{\bzeta_1}{\eps}, \frac{\bzeta_2}{\eps}\big) \big]
= \frac{k_0^2}{4(2\pi)^2}  \phi^\eps (\bzeta_1)  \int  \hat{C}(\bk) \\
&& \times \bigg[ -2  \phi^\eps (\bzeta_2) 
\tilde{B} \big( z, \frac{\bxi_2+\bxi_1}{2}, \frac{\bxi_2-\bxi_1}{2} ,\frac{\bzeta_1}{\eps}, \frac{\bzeta_2}{\eps}\big)\\
&& 
+  \phi^\eps (\bzeta_2) 
\tilde{B} \big( z, \frac{\bxi_2+\bxi_1}{2}-\bk, \frac{\bxi_2-\bxi_1}{2} ,\frac{\bzeta_1}{\eps}, \frac{\bzeta_2}{\eps}\big)
e^{ i \frac{z}{\eps k_0} \bk \cdot (\bzeta_2+\bzeta_1)} \\
&& 
+  \phi^\eps (\bzeta_2-\bk) 
\tilde{B} \big( z, \frac{\bxi_2+\bxi_1-\bk}{2} , \frac{\bxi_2-\bxi_1+\bk }{2} ,\frac{\bzeta_1}{\eps}, \frac{\bzeta_2-\bk}{\eps}\big)
e^{ i \frac{z}{\eps k_0} \bk \cdot (\bxi_2+\bzeta_1)} \\
&& 
+  \phi^\eps (\bzeta_2) 
\tilde{B} \big( z, \frac{\bxi_2+\bxi_1}{2} , \frac{\bxi_2-\bxi_1}{2} -\bk ,\frac{\bzeta_1}{\eps}, \frac{\bzeta_2}{\eps}\big)
e^{ i \frac{z}{\eps k_0} \bk \cdot (\bzeta_2-\bzeta_1)} \\
&& 
+  \phi^\eps (\bzeta_2-\bk) 
\tilde{B} \big( z, \frac{\bxi_2+\bxi_1+\bk}{2} , \frac{\bxi_2-\bxi_1-\bk }{2} ,\frac{\bzeta_1}{\eps}, \frac{\bzeta_2-\bk}{\eps}\big)
e^{ i \frac{z}{\eps k_0} \bk \cdot (\bxi_2-\bzeta_1)} \\
&& 
-  \phi^\eps (\bzeta_2-\bk) 
\tilde{B} \big( z, \frac{\bxi_2+\bxi_1-\bk}{2} , \frac{\bxi_2-\bxi_1-\bk }{2} ,\frac{\bzeta_1}{\eps}, \frac{\bzeta_2-\bk}{\eps}\big)
e^{ i \frac{z}{\eps k_0} \big( \bk \cdot (\bzeta_2+\bxi_2) - |\bk|^2\big)}  \\
&& 
-  \phi^\eps (\bzeta_2+\bk) 
\tilde{B} \big( z, \frac{\bxi_2+\bxi_1-\bk}{2} , \frac{\bxi_2-\bxi_1-\bk }{2} ,\frac{\bzeta_1}{\eps}, \frac{\bzeta_2+\bk}{\eps}\big)
e^{ i \frac{z}{\eps k_0} \big( \bk \cdot (\bzeta_2-\bxi_2) + |\bk|^2\big)} 
\bigg] d\bk
.
\end{eqnarray*}
The first, second and fourth terms will be canceled by the corresponding terms in $S_{B}^\eps$.
The other terms are negligible in the sense of (\ref{eq:lem2R}).
Indeed, for instance, the third term satisfies (using the change of variables $(\bzeta_1,\bxi_1,\bzeta_2)
\to ( \balpha = \bzeta_1/\eps,\bxi=\bxi_1-\bk, \bzeta= (\bzeta_2-\bk)/\eps)$):
\begin{eqnarray*}
&& 
\iint \bigg| \int_0^z dz' \int d\bk \hat{C}(\bk)
\phi^\eps (\bzeta_2-\bk)\phi^\eps (\bzeta_1) \\
&& \quad \times
\tilde{B} \big( z', \frac{\bxi_2+\bxi_1-\bk}{2} , \frac{\bxi_2-\bxi_1+\bk }{2} ,\frac{\bzeta_1}{\eps}, \frac{\bzeta_2-\bk}{\eps}\big)
e^{ i \frac{z'}{\eps k_0} \bk \cdot (\bxi_2+\bzeta_1)} 
\bigg| d\bxi_1 d\bxi_2 d\bzeta_1 d\bzeta_2\\
&&  \leq 
\iint \bigg| \int_0^z dz' \hat{C}(\bk)
\phi^1(\balpha) \phi^1(\bzeta) \tilde{B}\big( z', \frac{\bxi_2+\bxi}{2} , \frac{\bxi_2-\bxi}{2} ,
\balpha,\bzeta\big)
 e^{i \frac{z'}{k_0} \bk \cdot \balpha}
 e^{i \frac{z'}{k_0} \frac{\bk \cdot \bxi_2}{\eps}}
\bigg| d\bk d\bxi d\bxi_2 d\balpha d\bzeta .
\end{eqnarray*}
From Lemma \ref{lem:tech1} this term goes to zero as $\eps \to 0$.
\qed
\end{proof}

We can now state and prove the lemma that gives the statement of Proposition \ref{prop:sci1}.

\begin{lemma}
\label{lem:lem3R}%
For any $Z>0$ 
\begin{equation}
\label{eq:lem3R}
\sup_{z \in [0,Z]} \| R^\eps(z,\cdot,\cdot,\cdot,\cdot)\|_{L^1(\RR^2 \times \RR^2 \times \RR^2\times \RR^2)}  \stackrel{\eps \to 0}{\longrightarrow} 0   .
\end{equation}
\end{lemma}
\begin{proof}
We have for any $z$
$$
\big\| \big[{\cal L}^\eps_{z } R^\eps\big](z,\cdot,\cdot,\cdot,\cdot) \|_{L^1} \leq 2 k_0^2 C({\bf 0}) 
\big\| R^\eps (z,\cdot,\cdot,\cdot,\cdot) \|_{L^1}  .
$$
Therefore using the integral version of (\ref{eq:tildeReps2}) we obtain
$$
\big\|  R^\eps (z,\cdot,\cdot,\cdot,\cdot) \|_{L^1} \leq 
 2 k_0^2 C({\bf 0}) \int_0^z \big\|  R^\eps (z',\cdot,\cdot,\cdot,\cdot) \|_{L^1} dz'
+   \Big\| \int_0^z 
S^\eps (z',\cdot,\cdot,\cdot,\cdot) dz'\Big\|_{L^1}  .
$$
Using Lemma \ref{lem:lem2R} and Gronwall's lemma gives  the desired result.
\qed
\end{proof}

Finally we state and prove the technical Lemma \ref{lem:tech1} that was needed in the proof of Lemma \ref{lem:lem2R}.

\begin{lemma}
\label{lem:tech1}%
Let $m$ be a positive integer and $F\in {\cal C}( [0,Z],  L^1( \RR^m \times \RR^2 \times \RR^2))$.
For any $Z>0$ we have
\begin{equation}
\label{eq:lemtech1}
\sup_{z \in [0,Z] } 
\iint \Big| \int_0^z 
F (z',\bu,\bv,\bw) \exp \big( i \frac{z'}{\eps} \bv \cdot \bw) dz' \Big| d\bu d\bv d\bw  \stackrel{\eps \to 0}{\longrightarrow} 0  .
\end{equation}
Let $m$ be a positive integer  and $F\in {\cal C}( [0,Z],  L^1( \RR^m \times \RR^2))$.
For any $Z>0$ we have
\begin{equation}
\label{eq:lemtech2}
\sup_{z \in [0,Z]} 
\iint \Big| \int_0^z 
F  (z',\bu,\bv) \exp \big( i \frac{z'}{\eps} | \bv |^2 ) dz'\Big| d\bu d\bv    \stackrel{\eps \to 0}{\longrightarrow} 0  .
\end{equation}
\end{lemma}

\begin{proof}
Let us denote
$$
\tilde{F}^\eps   (z,\bu,\bv,\bw) = F (z,\bu,\bv,\bw)
\exp \Big( i \bv\cdot \bw \frac{z}{ \eps}\Big).
$$
For any $\delta >0$ we introduce the domain in  $\RR^m \times \RR^2\times \RR^2$:
$$
\Omega_\delta= \big\{ (\bu,\bv,\bw)\in \RR^m \times \RR^2 \times \RR^2\, , \, |\bv\cdot \bw | \leq \delta\big\}  .
$$
Since
$$
\Big|
\int_0^z \tilde{F}^\eps   (z',\bu,\bv,\bw)  dz' 
\Big|
\leq \int_0^z |F  (z',\bu,\bv,\bw)| dz'  ,
$$
we obtain
\begin{equation}
\label{eq:lem2:R1}
\sup_{z \in [0,Z] }
\iint_{\Omega_\delta} 
\Big|
\int_0^z \tilde{F}^\eps  (z,\bu,\bv,\bw)  dz' 
\Big|d\bu d\bv d\bw 
\leq
 \int_0^Z \int_{\Omega_\delta} |F  (z',\bu,\bv,\bw) | d\bu d\bv d\bw  dz'   .
\end{equation}
For any positive integer $n $ we have
\begin{eqnarray*}
\Big| \int_0^z \tilde{F}^\eps   (z',\bu,\bv,\bw)  dz' 
-
\sum_{k=0}^{n-1}
\int_{\frac{k}{n}z}^{\frac{k+1}{n}z}
F  \big(\frac{kz}{n},\bu,\bv,\bw \big)   \exp \Big( i \bv\cdot \bw \frac{z'}{ \eps}\Big)  dz' 
\Big|\\
\leq \sum_{k=0}^{n-1}
\int_{\frac{k}{n}z}^{\frac{k+1}{n}z} \big| F  \big(z',\bu,\bv,\bw \big) -
F  \big(\frac{kz}{n},\bu,\bv,\bw\big) \big|dz'  .
\end{eqnarray*}
Since
$$
\Big| \int_{\frac{k}{n}z}^{\frac{k+1}{n}z}
 \exp\Big( i \bv\cdot \bw \frac{z'}{\eps}\Big) dz' \Big|
=
 \Big| 
 \frac{ \exp\Big( i \bv\cdot \bw \frac{z}{n  \eps}\Big)-1}{ i \bv\cdot \bw \frac{1}{ \eps}}
 \Big| \leq  \frac{2 \eps}{\delta}\quad  \mbox{ if }  \quad (\bu,\bv,\bw)\not\in \Omega_\delta   ,
$$
we obtain
\begin{eqnarray}
\nonumber
 \sup_{z \in [0,Z] }
\iint_{\Omega_\delta^c}
\Big| \int_0^z \tilde{F}^\eps (z',\bu,\bv,\bw) dz' 
\Big| d\bu d\bv d\bw
\leq 
\sup_{z \in [0,Z] }
\| F  (z,\cdot,\cdot,\cdot)\|_{L^1}  \frac{2n \eps}{\delta}
\\
+ Z
\sup_{z_1,z_2 \in [0,Z], \, |z_1-z_2| \leq Z/n}
 \big\| F (z_1,\cdot,\cdot,\cdot) -
F  (z_2,\cdot,\cdot,\cdot) \big\|_{L^1}  .  
\label{eq:lem2:R2}
\end{eqnarray}
If  we sum (\ref{eq:lem2:R1}) and (\ref{eq:lem2:R2}) and take the $\limsup$ in $\eps$
 then we find:
\begin{eqnarray*}
\nonumber
\limsup_{\eps \to 0} \sup_{ z \in [0,Z] }
\Big\| \int_0^z \tilde{F}^\eps(z',\cdot,\cdot,\cdot ) dz' \Big\|_{L^1}
\leq 
 \int_0^Z \iint_{\Omega_\delta} |F (z',\bu,\bv,\bw)| d\bu d\bv d\bw  dz'
\\
+ 
Z \sup_{z_1,z_2 \in [0,Z], \, |z_1-z_2| \leq Z/n}
 \big\| F (z_1,\cdot,\cdot,\cdot) -
F(z_2,\cdot,\cdot,\cdot) \big\|_{L^1}   .
\end{eqnarray*}
We then take the limit $\delta \to 0$ and $n \to \infty$ in the right-hand side to obtain
the first result of the Lemma  (using Lebesgue's dominated convergence theorem).

The proof of the second statement of the Lemma is similar with the domain
$$
\Omega_\delta= \big\{ (\bu,\bv)\in \RR^m \times \RR^2 \, , \, |\bv |^2  \leq \delta\big\}  .
$$
\qed
\end{proof}


\begin{thebibliography}{99}

\bibitem{andrews}
L. C. Andrews and R. L. Philipps, 
{\it Laser Beam Propagation Through Random Media},
SPIE Press, Bellingham, 2005.

\bibitem{bailly96}
F. Bailly, J.-F. Clouet, and J.-P. Fouque,
 Parabolic and white noise approximation for waves in random media,
SIAM J. Appl. Math. {\bf 56}  (1996), 1445-1470.

\bibitem{balMMS}
G. Bal,
 {On the self-averaging of wave energy in random media}, 
SIAM Multiscale Model. Simul.  {\bf 2} (2004), 398-420. 

\bibitem{bal}
G. Bal and O. Pinaud, 
Dynamics of wave scintillation in random media,
Comm. Partial Differential Equations {\bf 35} (2010), 1176-1235.

\bibitem{blomgren}  
P. Blomgren, G. Papanicolaou, and H. Zhao, 
Super-resolution in time-reversal acoustics, 
J. Acoust. Soc. Am. {\bf 111} (2002), 230-248. 

\bibitem{cartwright76}
N. D. Cartwright, 
A non-negative Wigner-type distribution,
Physica {\bf 83A} (1976), 210-212.

\bibitem{cheng09}
J. Cheng,
Ghost imaging through turbulent atmosphere,
{Opt. Express} {\bf 17} (2009), 7916-7917.

\bibitem{claerbout85}
J. F. Claerbout, 
{\it Imaging the Earth's Interior},
Blackwell Scientific Publications, Palo Alto, 1985.
 
\bibitem{dawson84} 
D. Dawson and G. Papanicolaou,  
{A random wave process},
Appl. Math. Optim. {\bf 12} (1984), 97-114.   

\bibitem{dehoopsolna09}
M. de Hoop and K. S\o lna,
Estimating a Green's function from ``field-field'' correlations in a random medium, 
SIAM J. Appl. Math. {\bf 69} (2009), 909-932. 

\bibitem{dehoopgarnier13}
M. de Hoop, J. Garnier, S. F. Holman, and K. S\o lna,
Retrieval of a Green's function with reflections from partly coherent waves generated by a wave packet using cross correlations,
SIAM J. Appl. Math.  {\bf 73} (2013), 493-522.

\bibitem{fann06}
A. C. Fannjiang, 
Self-averaging radiative transfer for parabolic waves,
C. R. Acad. Sci. Paris, Ser. I {\bf 342} (2006), 109-114.

\bibitem{fann09}
A. C. Fannjiang,
Introduction to propagation, time reversal and imaging in random media,
in Multi-Scale Phenomena in Complex Fluids - Modeling, Analysis and Numerical Simulation, 
T.Y Hou, C. Liu and J.G. Liu, eds., World Scientific, 2009. 

\bibitem{fann05b}
{A. Fannjiang and K. S\o lna}, 
{Superresolution and duality for time-reversal of waves in random media},
Phys. Lett. A {\bf  352} (2005), 22-29.

\bibitem{fante75}
R. L. Fante,
Electromagnetic beam propagation in turbulent media,
Proc. IEEE {\bf 63} (1975), 1669-1692.

\bibitem{feizulin}
{Z. I. Feizulin and Yu. A. Kravtsov},
{Broadening of a laser beam in a turbulent medium},
Radio Quantum Electron. {\bf 10} (1967),  33-35.


\bibitem{fink}
M. Fink, {Time-reversed acoustics}, Scientific American, November issue (1999), 91-97.
 
\bibitem{book1}
J.-P. Fouque, J. Garnier, G. Papanicolaou,  and K. S\o lna,
{\em  Wave Propagation and Time Reversal in Randomly Layered Media},
 Springer,  New York, 2007. 

\bibitem{fps}
J.-P. Fouque,  G. Papanicolaou,  and Y. Samuelides,
{Forward and Markov approximation: the strong-intensity-fluctuations regime revisited},
Waves in Random Media {\bf 8} (1998),   303-314. 

\bibitem{furutsu72}
K. Furutsu, 
Statistical theory of wave propagation in a random medium and the irradiance distribution function,
J. Opt. Soc. Am. {\bf 62} (1972), 240-254.

\bibitem{furutsu73}
K. Furutsu and Y. Furuhama,
Spot dancing and relative saturation phenomena of irradiance scintillation of optical beams in a random medium, 
Optica {\bf 20} (1973), 707-719.

\bibitem{GP}
J. Garnier and G. Papanicolaou,
Passive sensor imaging using cross correlations of noisy signals in a scattering medium,
{SIAM J. Imaging Sciences} {\bf 2} (2009), 396-437. 
 
\bibitem{garniers0}
J. Garnier and K. S\o lna,
Random backscattering in the parabolic scaling,
J. Stat. Phys. {\bf 131} (2008),  445-486. 

\bibitem{garniers1} 
J. Garnier and K. S\o lna,
Coupled paraxial wave equations in random media in the white-noise regime,
Ann. Appl. Probab. {\bf 19} (2009), 318-346. 

\bibitem{garniers2}
J. Garnier and K. S\o lna,
Scaling limits for wave pulse transmission and reflection operators,
Wave Motion {\bf 46} (2009), 122-143.

\bibitem{garniers3}
J. Garnier and K. S\o lna,
Scintillation in the white-noise paraxial regime,
Comm. Partial Differential Equations {\bf 39} (2014), 626-650.

\bibitem{gerard}
P. G\'erard, P. A. Markowich, N. J. Mauser, and F. Poupaud,
Homogenization limits and Wigner transforms,
Comm. Pure Appl. Math. {\bf 50} (1997), 323-379.

\bibitem{gozani}
J. Gozani,
Numerical solution of the fourth-order coherence function of a plane 
wave propagating in a two-dimensional Kolmogorovian medium, 
J. Opt. Soc. Am. A {\bf 2} (1985),  2144-2151.

\bibitem{ishimaru}
A. Ishimaru, 
{\em Wave Propagation and Scattering in Random Media},
Academic Press, San Diego, 1978.

\bibitem{katz12}
O. Katz, E. Small, and Y. Silberberg,
Looking around corners and through thin turbid layers in real time with scattered incoherent light,
Nature Photon. {\bf 6} (2012), 549-553.

\bibitem{ryzhikCMP}
T. Komorowski,  S. Peszat, and L. Ryzhik,
Limit of fluctuations of solutions of Wigner equation, 
Commun. Math. Phys. {\bf 292} (2009), 479-510. 

\bibitem{ryzhikDC}
T. Komorowski and L. Ryzhik,
Fluctuations of solutions to Wigner equation with an Ornstein-Uhlenbeck potential, 
Discrete and Continuous Dynamical Systems-Series B {\bf 17}  (2012),
871-914.


\bibitem{li10}
C. Li, T. Wang, J. Pu, W. Zhu, and R. Rao,
Ghost imaging with partially coherent light radiation through turbulent atmosphere,
{Appl. Phys. B} {\bf 99} (2010), 599-604.

\bibitem{manfredi00}
G. Manfredi and M. R. Feix,
Entropy and Wigner functions,
Phys. Rev. E {\bf 62} (2000), 4665-4674.

\bibitem{Mao2012} 
Y. Mao and J. Gilles,
Non rigid geometric distortions correction - Application to atmospheric turbulence stabilization,
Inverse Problems and Imaging {\bf 6}  (2012), 531-546. 

\bibitem{mosk12}
A. P. Mosk, A. Lagendijk, G. Lerosey, and M. Fink, 
Controlling waves in space and time for imaging and focusing in complex media,
Nature Photon. {\bf 6} (2012), 283-292.

\bibitem{miyahara82}
Y. Miyahara,
Stochastic evolution equations and white noise analysis,
Carleton Mathematical Lecture Notes {\bf 42}, Ottawa, Canada,  (1982),  1-80. 

\bibitem{PRS04}
G. Papanicolaou, L. Ryzhik, and K. S{\o}lna,
Statistical stability in time reversal,
SIAM J.  Appl. Math.  {\bf 64}  (2004), 1133-1155.

\bibitem{PRS07}
G. Papanicolaou, L. Ryzhik, and K. S{\o}lna,
Self-averaging from lateral diversity in the {It\^o-Schr\"odinger} equation,
SIAM Multiscale Model. Simul. {\bf 6}  (2007), 468-492.

\bibitem{popoff14}
S. M.  Popoff, A. Goetschy, S. F. Liew, A. D. Stone, and H. Cao,
Coherent control of total transmission of light through disordered media,
Phy. Rev. Lett. {\bf 112} (2014),  133903.

\bibitem{popoff10}
S. Popoff,	 G. Lerosey, M. Fink, A. C. Boccara, and S. Gigan,
Image transmission through an opaque material,
Nature Commun. {\bf 1} (2010), 1-5.

 \bibitem{reed62}
I. S. Reed,
On a moment theorem for complex Gaussian processes,
IRE Trans. Inform. Theory {\bf IT-8} (1962), 194-195.

\bibitem{ryzhik}
L. Ryzhik, J. Keller, and G.Papanicolaou
Transport equations for elastic and other waves in random media,
Wave Motion {\bf 24} (1996), 327-370.

\bibitem{shapiro12}
J. H. Shapiro and R. W. Boyd,
The physics of ghost imaging,
{Quantum Inf. Process.} {\bf 11} (2012), 949-993.

\bibitem{strohbehn}
J. W. Strohbehn, ed., 
{\it Laser Beam Propagation in the Atmosphere},
 Springer, Berlin, 1978.

\bibitem{tappert}
F. Tappert, 
The parabolic approximation method,
in {\it Wave Propagation and Underwater Acoustics},
J. B. Keller and  J. S.  Papadakis, eds.,  224-287,
Springer, Berlin (1977).
 
 \bibitem{tatarski71}
V. I. Tatarskii, {\it The Effect of Turbulent Atmosphere on Wave Propagation},
U.S. Department of Commerce, TT-68-50464, Springfield, 1971.
 
 \bibitem{tatarskii}
V. I. Tatarskii, A. Ishimaru, and V. U. Zavorotny, eds., 
{\it Wave Propagation in Random Media (Scintillation)}, 
SPIE Press, Bellingham, 1993.

\bibitem{Tofsted2011} 
D. H. Tofsted,
Reanalysis of turbulence effects on short-exposure passive imaging,
Opt. Eng. {\bf 50} (2011), 016001.

\bibitem{uscinski}
B. J. Ucsinski, 
Analytical solution of the fourth-moment equation and interpretation as a set of phase screens,
J. Opt. Soc. Am. A {\bf 2} (1985),  2077-2091.

\bibitem{valley}
G. C. Valley and D. L. Knepp,
Application  of joint Gaussian  statistics to interplanetary  scintillation,
J. Geophys. Res. {\bf 81} (1976), 4723-4730.

\bibitem{webb2}
J. A. Newman and K. J. Webb, 
Fourier magnitude of the field incident on a random scattering medium from spatial speckle intensity correlations,
Opt. Lett. {\bf 37} (2012), 1136-1138.

\bibitem{webb3}
J. A. Newman and K. J. Webb,
Imaging optical fields through heavily scattering media,
Phys. Rev. Lett. {\bf 113} (2014), 263903.

\bibitem{vellekoop10}
I. M. Vellekoop, A. Lagendijk, and A. P. Mosk,
Exploiting disorder for perfect focusing,
Nature Photon. {\bf 4} (2010), 320-322.

 \bibitem{vellekoop07}
I. M. Vellekoop and A. P. Mosk,
Focusing coherent light through opaque strongly scattering media,
Opt. Lett. {\bf 32} (2007), 2309-2311.

\bibitem{vellekoop08}
I. M. Vellekoop and A. P. Mosk,
Universal optimal transmission of light through disordered materials,
Phys. Rev. Lett. {\bf 101} (2008), 120601.

\bibitem{whitman}
A. M. Whitman and M. J. Beran,
Two-scale solution for atmospheric scintillation,
J. Opt. Soc. Am. A {\bf 2} (1985),  2133-2143.

\bibitem{ya}
I. G. Yakushkin, 
Moments of field propagating in randomly inhomogeneous medium
in the limit of saturated  fluctuations, 
Radiophys. Quantum Electron. {\bf 21} (1978), 835-840.

\end{thebibliography}
\end{document}